\documentclass{amsart}

\usepackage{layout}
\usepackage{amsmath, amsthm}
\usepackage{amssymb, amsrefs}
\usepackage{stmaryrd}
\usepackage{graphicx}
\usepackage{setspace}
\usepackage{textcomp}
\usepackage{esvect}
\usepackage[letterpaper]{geometry}
\usepackage{arydshln}
\usepackage[all]{xy}
\usepackage{etex}

\usepackage[dvipsnames]{xcolor}

\usepackage{epsfig}
\usepackage{latexsym}
\usepackage{amsthm}

\usepackage{mathrsfs}

\usepackage[colorlinks,citecolor=blue]{hyperref}

\usepackage[latin1]{inputenc}

\usepackage{tikz-cd}
\usetikzlibrary{graphs,graphs.standard}
\usepackage{pgfplots}
\usetikzlibrary{matrix,arrows,decorations.pathmorphing}
\usetikzlibrary{shapes.geometric}

\newtheorem{theorem}{Theorem}[section]
\newtheorem{lemma}[theorem]{Lemma}
\newtheorem{proposition}[theorem]{Proposition}

\newtheorem{corollary}[theorem]{Corollary}

\theoremstyle{definition}
\newtheorem{definition}[theorem]{Definition}
\newtheorem{example}[theorem]{Example}

\newtheorem{obs}[theorem]{Observation}

\theoremstyle{definition}
\numberwithin{equation}{section}

\tikzset{
    ncbar angle/.initial=90,
    ncbar/.style={
        to path=(\tikztostart)
        -- ($(\tikztostart)!#1!\pgfkeysvalueof{/tikz/ncbar angle}:(\tikztotarget)$)
        -- ($(\tikztotarget)!($(\tikztostart)!#1!\pgfkeysvalueof{/tikz/ncbar angle}:(\tikztotarget)$)!\pgfkeysvalueof{/tikz/ncbar angle}:(\tikztostart)$)
        -- (\tikztotarget)
    },
    ncbar/.default=0.5cm,
}

\tikzset{square left brace/.style={ncbar=0.5cm}}
\tikzset{square right brace/.style={ncbar=-0.5cm}}

\tikzset{round left paren/.style={ncbar=0.5cm,out=120,in=-120}}
\tikzset{round right paren/.style={ncbar=0.5cm,out=60,in=-60}}

\DeclareMathOperator{\id}{id}

\DeclareMathOperator{\Hom}{Hom}
\DeclareMathOperator{\Obj}{Obj}

\def \Gph {\mathsf{Gph}}
\def \fGph {\mathsf{FGph}}
\def \HoGph {\mathsf{hFGph}}
\def \sGph {\mathsf{stGph}}
\def \into {\hookrightarrow}

\def \con {\sim}
\definecolor{laura}{rgb}{.4, 0, .6}

\begin{document}

\title{A Homotopy Category for Graphs
%\thanks{}
}
% Grants or other notes about the article that should go on the front
% page should be placed within the \thanks{} command in the title
% (and the %-sign in front of \thanks{} should be deleted)
%
% General acknowledgments should be placed at the end of the article.



\author{Tien Chih         \and
        Laura Scull %etc.
}

\maketitle

\begin{abstract}
We show that the category of graphs has the structure of a 2-category with homotopy as the 2-cells.  We then develop an explicit description of homotopies for finite graphs, in terms of what we call `spider moves'.    We then create a  category by modding out by the 2-cells of our 2-category, and use the spider moves to show that for finite  graphs, this category is a homotopy category in the sense that it  satisfies the universal property for localizing homotopy equivalences.  We then show that finite stiff graphs form a skeleton of this homotopy category. 
%\keywords{graph homomorphism \and  homotopy \and 2-category \and homotopy category \and  skeleton of a category}
% \subclass{05C60 \and 55U35  \and 18D05 }
\end{abstract}

\section{Introduction}\label{S:Intro}
Homotopy traditionally studies continuous transformations of spaces and maps between them.  Translating such a fundamentally continuous concept into a discrete setting such as graphs can be approached in several ways.  The first strategy used was to create a 'Hom complex', a polyhedral complex which represents information about the morphisms between two graphs.  This complex can then be turned into a topological space, and the homotopy of this space encodes information about homotopy of graphs  \cite{HomTG, MoreFolding, Kosolov1,Kosolov2, Kosolov3, KosolovShort, ProofLovasz}.  More recently, Dochtermann has shown that it is possible to define a homotopy for graphs, called $\times$- homotopy, using only categorical constructions inside of graphs, and get the same homotopy theory as that provided by simplicial techniques  \cite{Docht1}. Others have since developed results strictly within the graph category \cite {Docht2, Demitri, Droz, NoModel}.    We follow this second approach and study $\times$-homotopy, working strictly with graphs and discrete constructions.

In topological spaces, the existence of the homotopies and their structure gives rise to a 2-categorical structure on spaces, in which the homotopies form 2-cells.  In this paper, we show that the category of graphs also has the structure of a 2-category with homotopies of morphisms as the 2-cells, and verify the necessary conditions.  We then develop an explicit description of homotopy for finite graphs, based around our notion of 'spider moves'.  Our spider moves can be seen as a generalization of the idea of folds, which have been linked to  homotopy of graphs  by \cite{BonatoCaR,GMDG,HN2004}.    We then define a quotient category of our 2-category, and use our spider moves to show that this forms a homotopy category for finite graphs  in the sense that it  satisfies the universal property for localization of homotopy equivalences.  Such a  localization is often created via a Quillen model category, which offers extra structure for working with the homotopy category that is created.  The existence  of model structures for the category of graphs has been studied by \cite{Droz, BoxHomotopy}, and a number of  different model structures have been defined  which localize with respect to various notions of graph homotopy.  They do not produce a model structure for the $\times$-homotopy that we are studying, and in fact it is shown in \cite{NoModel} that no such model stucture exists that has some subclass of inclusions as cofibrations, ruling out the most natural attempt to generalize existing model structures in other areas of mathematics.    Here, we simply provide a direct construction of the localized homotopy category without a model structure.    To give some handle on the structure of the localized category, we show that the subcategory formed by stiff graphs forms a skeleton of our homotopy category, and hence the stiff graphs give canonical representatives for finite graphs up to homotopy.

We begin in Section \ref{S:BD} by reviewing the basic definitions and properties of the graph category, including products, exponential objects and walks and their concatenations following \cite{Demitri, AGT, HN2004, Docht1}.  In Section \ref{S:HT}, we  establish that the category of graphs forms a 2-category.  In Section \ref{S:Struct} we give a concrete description of the structure of a homotopy of graph morphisms, showing that a  homotopy with finite domain can be broken down into a sequence of simple 'spider moves' which move only one vertex at a time.    In Section \ref{S:HoCat} we  use our spider moves from Section \ref{S:Struct} to show that the quotient of the 2-category constructed is a categorical homotopy category for finite graphs in the sense that it  satisfies the universal property for localization of homotopy equivalences.  In Section \ref{S:skel}, we show that the finite stiff graphs form a skeleton for the new homotopy category  and briefly discuss what can be said about  the structure of this skeleton in the absence of any model categorical infrastructure.

\section{Background}\label{S:BD}
In this section, we give background definitions and notations.   We include some basic results which seem like they should be standard, but we were unable to find specific references in the literature, so we include them here for completeness.  We will use standard graph theory definitions and terminology following \cite{Bondy, AGT, HN2004}, and category theory definitions and terminology from \cite{riehlCTIC, Mac}.

\subsection{The Graph Category} We work in the category $\Gph$ of finite undirected graphs, where we allow at most one edge connecting any pair of vertices.  We do allow a (single) loop connecting a vertex to itself.    

\begin{definition} \cite{HN2004} The category of graphs $\Gph$ is defined by: \begin{itemize} \item An object is a  graph $G$,  consisting of a set of vertices $V(G) = \{ v_{\lambda}\}$ and a set $E(G)$ of edges connecting them, where each edge is given by an unordered set of two vertices.      If two vertices are connected by an edge, we will use notation  $v_1 \con v_2 \in E(G)$, or just $v_1 \con v_2$ if the parent graph is clear.   
\item 
An arrow in the category   $\Gph$ is a graph morphism $f:  G \to H$.  Specifically, this is  given by a set map $f:  V(G) \to V(H)$ such that  if $v_1 \con v_2 \in E(G)$ then $f(v_1) \con f(v_2) \in E(H)$.  \end{itemize} \end{definition}  

We will work in this category throughout this paper, and assume that 'graph' always refers to an object in $\Gph$.    When we have an invertible graph morphism $f:  G \to H$  we will say that $G$ and $H$ are isomoprhic and write $G \cong H$.

\begin{definition}\label{D:im}\cite{Demitri} 
Given a homomorphism $f:G\to H$, we define the image $ Im(f)$ to be the subgraph of  $H$ where $V(Im(f))=\{f(v):v\in G\} $ and $ E(Im(f))=\{f(v)\con f(w):v \con w\in E(G)\}$.  Thus we specifically consider $Im(f)$ to contain only edges which are images of edges in $G$.  
\end{definition}

\begin{definition}\label{D:prod} \cite{HN2004, Demitri} The (categorical) {\bf product graph} $G \times H$ is defined by: \begin{itemize}
\item  A vertex is a  pair $(v, w)$ where $v \in V(G)$ and $w \in V(H)$.
\item An edge is defined by $(v_1, w_1) \con (v_2, w_2) \in E(G \times H) $ for $v_1 \con v_2 \in E(G)$ and $w_1 \con w_2 \in E(H)$.   \end{itemize}
\end{definition}

\begin{example}\label{E:examprod}

Let $G$ be the graph on two adjacent looped vertices:  $V(G) = \{ 0, 1\}$ and $E(G) = \{ 0\con 0,1 \con 1, 0 \con 1\}$.  Let  $H= K_2 $ with $V(H) = \{ a, b\} $ and $E(H) = \{ a\con b\}$.   Then $G \times H$ is isomorphic to the cyclic graph $C_4$:

$$\begin{tikzpicture}

\draw[fill] (1,.5) circle (2pt);
\draw (1,.5) --node[below]{$a$} (1,.5);
\draw[fill] (2,.5) circle (2pt);
\draw (2,.5) --node[below]{$b$} (2,.5);
\draw (1,.5) -- (2,.5);

\draw[fill] (.5,1) circle (2pt);
\draw (.5,1) --node[left]{0} (.5,1);
\draw[fill] (.5,2) circle (2pt);
\draw (.5,2) --node[left]{1} (.5,2);
\draw (.5,1) -- (.5,2);
\draw (.5,1)  to[in=-40,out=220,loop, distance=.7cm] (.5,1);
\draw (.5,2)  to[in=-220,out=40,loop, distance=.7cm] (.5,2);

\draw[fill] (1,1) circle (2pt);
\draw (1,1) --node[below]{\tiny $(0,a)$} (1,1);
\draw[fill] (1,2) circle (2pt);
\draw (2,1) --node[below]{\tiny $(0,b)$} (2,1);
\draw[fill] (2,1) circle (2pt);
\draw (1,2) --node[above]{\tiny $(1,a)$} (1,2);
\draw[fill] (2,2) circle (2pt);
\draw (2,2) --node[above]{\tiny $(1,b)$} (2,2);

\draw (1,1) -- (2,2);
  \draw (1,1) -- (2,1);
\draw (1,2) -- (2,1);
\draw (1,2) -- (2,2);

\end{tikzpicture}$$

\end{example}

\begin{lemma} \cite{HN2004} \label{L:incl}  If $w \in V(H)$ is looped, i.e. $w \con w \in E(H)$, then there is an inclusion $G \to G \times H$ given by $v \rightarrow (v, w)$ which is a graph morphism. 
\end{lemma}

\begin{proof} If  $w$ is looped then  $v \con v'$ in $G$ if and only if $(v, w) \con (v', w)$ in $G \times H$.    Thus  the subgraph $G \times \{w\}$ is isomorphic to $G$. 

\end{proof}

\begin{definition}\label{D:exp} \cite{Docht1}
The {\bf exponential graph} $ H^G$ is defined by: \begin{itemize} \item  A vertex in  $V(H^G) $ is a {\bf set} map $V(G) \to V(H)$ [not necessarily a graph morphism]. \item There is an edge   $f \con g$ if  whenever $v_1 \con v_2 \in E(G)$, then $f(v_1) \con g(v_2) \in E(H)$. \end{itemize}     \end{definition} 

\begin{example}\label{E:examexp}
Let $G$ and $H$ be the following graphs:

$$
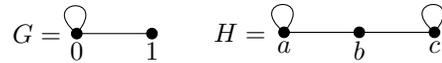
\begin{tikzpicture} 
\node at (-.5,0){$G=  $\,};

\draw[fill] (0,0) circle (2pt);
\draw (0,0) --node[below]{0} (0,0);
\draw (0,0)  to[in=50,out=140,loop, distance=.7cm] (0,0);

\draw[fill] (1,0) circle (2pt);
\draw (1,0) --node[below]{1} (1,0);

\draw (0,0) -- (1,0);
\end{tikzpicture}  
\hspace{.5cm}
\begin{tikzpicture} 
\node at (-.5,0){$H=$  \, };

\draw[fill] (0,0) circle (2pt);
\draw (0,0) --node[below]{$a$} (0,0);
\draw (0,0)  to[in=50,out=140,loop, distance=.7cm] (0,0);

\draw[fill] (1,0) circle (2pt);
\draw (1,0) --node[below]{$b$} (1,0);

\draw[fill] (2,0) circle (2pt);
\draw (2,0) --node[below]{$c$} (2,0);
\draw (2,0)  to[in=50,out=140,loop, distance=.7cm] (2,0);

\draw (0,0) -- (1,0) -- (2,0);
\end{tikzpicture}$$
Then the exponential graph $H^G$ is illustrated below, where the row indicates the image of $0$ and the column the image of $1$.  So for example the vertex in the $(a, c)$ spot represents the vertex map $f(0)=a, f(1)=c$. 
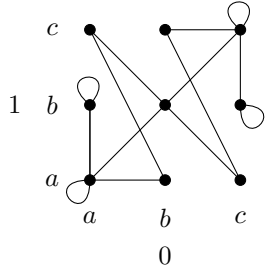
$$\begin{tikzpicture}

\node at (0,-.5){$a$};
\node at (1,-.5){$b$};
\node at (2,-.5){$c$};
\node at (1,-1){$0$};

\node at (-.5, 0){$a$};
\node at (-.5, 1){$b$};
\node at (-.5, 2){$c$};
\node at (-1, 1){$1$};

\draw[fill] (0,0) circle (2pt);
\draw[fill] (1,0) circle (2pt);
\draw[fill] (2,0) circle (2pt);
\draw[fill] (0,1) circle (2pt);
\draw[fill] (1,1) circle (2pt);
\draw[fill] (2,1) circle (2pt);
\draw[fill] (0,2) circle (2pt);
\draw[fill] (1,2) circle (2pt);
\draw[fill] (2,2) circle (2pt);

\draw (0,0)  to[in=180,out=270,loop, distance=.7cm] (0,0);
\draw (0,1)  to[in=50,out=140,loop, distance=.7cm] (0,1);
\draw (2,1)  to [in=270,out=0,loop, distance=.7cm](2,1);
\draw (2,2)  to[in=50,out=140,loop, distance=.7cm] (2,2);

\draw (0,0) -- (0,1);
\draw (0,0) -- (1,0);
\draw (0,0) -- (1,1);
\draw (0,2) -- (1,0);
\draw (0,2) -- (1,1);
\draw (0,0) -- (0,1);
\draw (1,1) -- (2,0);
\draw (1,1) -- (2,2);
\draw (1,2) -- (2,0);
\draw (1,2) -- (2,2);
\draw (2,1) -- (2,2);

%\node at (1,2.5){$H^G=$};

\end{tikzpicture}$$
\end{example}

\begin{obs} \label{O:morph}  If $f$ is looped in $H^G$,  this means exactly that if  $v_1 \con v_2 \in E(G)$, then $f(v_1) \con f(v_2) \in E(H)$. Thus a set map $f:  V(G) \to V(H)$ is a graph morphism if and only if $f \con f \in E(G^H)$.  \end{obs}

\begin{lemma} \label{L:comp1}  If $\phi:  H \to K$ is a graph morphism and $f \con g  \in E( H^G)$ then  $\phi f \con \phi g \in E(K^G)$.     So $\phi$ induces a graph morphism $\phi_*:  H^G \to K^G$.  

\end{lemma}

\begin{proof}  Suppose that $f \con g \in E(G^H)$.  So for any $v_1 \con v_2 \in E(G)$,  we know that $f(v_1) \con g(v_2) \in E(H)$.  Since $\phi$ is a graph morphism,  $\phi(f(v_1)) \con  \phi(g(v_2)) \in E(K)$.  So $\phi f \con \phi g$.    

\end{proof}

\begin{lemma} \label{L:comp2}  If $\psi:  K \to G$ is a graph morphism and $f \con g \in  E(H^G)$ then  $ f \psi \con g \psi \in E(H^K)$.   So $\psi$ induces a graph morphism $\psi^*:  H^G \to H^K$.  

\end{lemma}

\begin{proof}  Suppose that $v_1 \con v_2 \in E(K)$;  then we know that  $\psi(v_1) \con \psi(v_2) \in  E(G)$.  Since  $f \con g$ in $H^G$,  $f(\psi(v_1) )\con g(\psi((v_2))$.  So$ f \psi \con g \psi$.   

\end{proof}

\begin{proposition}

 \label{P:adj} \cite{Docht1}  The category $\Gph$ is cartesian closed.  In particular, we have a bijection   $$\Gph(G \times H, K) \cong \Gph(G, K^H)$$ \end{proposition} 
%%%%%%%%%%%%%%%%%%%%%%%%%%%%%%%%%%%%%%%%%%%%%%%%%%%%%%%%%%%%%%%%%%%%%%%%%%

\subsection{Walks and Concatenation}
\begin{definition} Let $P_n$ be the path graph with $n+1$ vertices $\{ 0, 1, \dots, n\}$ such that $i \con i+1$.  Let $I_n^{\ell}$ be the looped path graph with  $n+1$ vertices   $\{ 0, 1, \dots, n\}$ such that $i \con i$ and  $i \con {i+1}$.  

$$ 
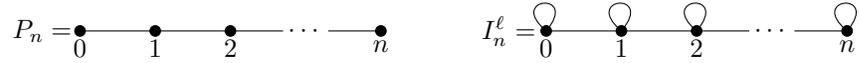
\begin{tikzpicture}
\node at (-.5,0){$P_n = $};
\draw[fill] (0,0) circle (2pt);
\draw (0,0) --node[below]{0} (0,0);
%\draw (0,0)  to[in=50,out=140,loop, distance=.7cm] (0,0);

\draw[fill] (1,0) circle (2pt);
\draw (1,0) --node[below]{1} (1,0);
%\draw (1,0)  to[in=50,out=140,loop, distance=.7cm] (1,0);

\draw[fill] (2,0) circle (2pt);
\draw (2,0) --node[below]{2} (2,0);
%\draw (2,0)  to[in=50,out=140,loop, distance=.7cm] (2,0);

\node at (3,0){$\cdots$}  ;

\draw[fill] (4,0) circle (2pt);
\draw (4,0) --node[below]{$n$} (4,0);
%\draw (4,0)  to[in=50,out=140,loop, distance=.7cm] (4,0);

\draw(0,0) -- (1,0);
\draw(1,0) -- (2,0);
\draw(2,0) -- (2.7,0);
\draw(3.3,0) -- (4,0);

\end{tikzpicture}
\hspace{1cm} \begin{tikzpicture}
\node at (-.5,0){$I_n^{\ell} = $};
\draw[fill] (0,0) circle (2pt);
\draw (0,0) --node[below]{0} (0,0);
\draw (0,0)  to[in=50,out=140,loop, distance=.7cm] (0,0);

\draw[fill] (1,0) circle (2pt);
\draw (1,0) --node[below]{1} (1,0);
\draw (1,0)  to[in=50,out=140,loop, distance=.7cm] (1,0);

\draw[fill] (2,0) circle (2pt);
\draw (2,0) --node[below]{2} (2,0);
\draw (2,0)  to[in=50,out=140,loop, distance=.7cm] (2,0);

\node at (3,0){$\cdots$}  ;

\draw[fill] (4,0) circle (2pt);
\draw (4,0) --node[below]{$n$} (4,0);
\draw (4,0)  to[in=50,out=140,loop, distance=.7cm] (4,0);

\draw(0,0) -- (1,0);
\draw(1,0) -- (2,0);
\draw(2,0) -- (2.7,0);
\draw(3.3,0) -- (4,0);

\end{tikzpicture}$$
\end{definition} 

\begin{definition}\label{D:path}  A  {\bf walk} in $G$ of length $n$ is a morphism $\alpha:  P_n \to G$.    A {\bf looped walk} in 
$G$ of length $n$ is a morphism $\alpha:  I_n^{\ell} \to G$. If $\alpha(v_0) = x$ and $\alpha(v_n ) = y$ we say $\alpha$ is a  walk [resp. looped walk]  from $x$ to $y$. \end{definition}

A walk can be described by a list of vertices $(v_0 v_1 v_2  \dots v_n)$ giving the images of the vertices   $\alpha(i) = v_i$, such that  $v_i \con v_{i+1}$.    Thus this definition agrees with the usual graph definition of walk.  In the looped case, since $i\con  i \in E(I_n^{\ell})$, we will have $v_i \con v_i$ and so a looped walk is simply a walk where all the vertices along the walk are looped.  

\begin{definition}\label{D:concat}   Given a walk $\alpha:  P_{n} \to G$ from $x$ to $y$,  and a walk $\beta:   P_{m} \to G$ from $y$ to $z$, we define the   {\bf concatenation of walks}  $\alpha * \beta:  P_{m+n} \to G $ by    $$ (\alpha*\beta) (i) = \begin{cases} \alpha(i) &  \textup{ if } i\leq n \\ 
\beta({i-n}) &  \textup{  if } n < i \leq n +m\\ 
\end{cases} $$ Since we are assuming that $\alpha(n )  = y = \beta(0)$,  $\alpha * \beta$ defines a length $n+m$ walk   from $x$ to $z$.  In vertex list form, the concatenation $(xv_1v_2 \dots v_{n-1}y)*(yw_1w_2 \dots w_{m-1}z) = (x v_1 v_2 \dots v_{n-1} y w_1 \dots w_{m-1} z).$
Contatenation of looped walks is defined in the same way.   \end{definition} 

\begin{example}\label{E:concat}  
Consider the graph below, and let $\alpha$ be a length 1 looped walk $(v_1v_2)$ and $\beta$ a length 2 looped walk $(v_2v_3v_4)$. 

$$
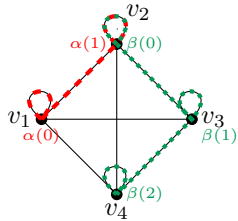
\begin{tikzpicture}% Image of alphas

\draw[fill] (0,0) circle (2pt);
\draw (0,0) --node[left] {$v_1$} (0,0);
\draw (0,0)  to[in=50,out=140,loop, distance=.7cm] (0,0);
\draw[fill] (1,1) circle (2pt);
\draw (1,1.2) --node[above right] {$v_2$} (1,1.2);
\draw (1,1)  to[in=50,out=140,loop, distance=.7cm] (1,1);
\draw[fill] (2,0) circle (2pt);
\draw (2,0)  to[in=50,out=140,loop, distance=.7cm] (2,0);
\draw (2,0) --node[right] {$v_3$} (2,0);
\draw[fill] (1,-1) circle (2pt);
\draw (1,-1)  to[in=50,out=140,loop, distance=.7cm] (1,-1);
\draw (1,-1) --node[below] {$v_4$} (1,-1);

\draw (0,0)--(1,1)--(2,0)--(1,-1)--(0,0);
\draw (0,0) -- (2,0);
\draw (1,1) -- (1,-1);

\draw[dashed, ultra thick, red] (0,0)--(1,1);
\draw[dotted, ultra thick, ForestGreen] (1,1)--(2,0)--(1,-1);

\draw[dashed, ultra thick, red] (0,0)  to[in=50,out=140,loop, distance=.7cm] (0,0);
\draw[dashed, ultra thick, red] (1,1)  to[in=50,out=140,loop, distance=.7cm] (1,1);
\draw[dotted, ultra thick, ForestGreen] (1,1)  to[in=50,out=140,loop, distance=.7cm] (1,1);

\draw[dotted, ultra thick, ForestGreen] (2,0)  to[in=50,out=140,loop, distance=.7cm] (2,0);
\draw[dotted, ultra thick, ForestGreen] (1,-1)  to[in=50,out=140,loop, distance=.7cm] (1,-1);

\draw[red] (0,0) --node[below] {\tiny{$\alpha(0)$}} (0,0);
\draw[red] (1,1) --node[left] {\tiny{$\alpha(1)$}} (1,1);

\draw[ForestGreen] (1,1) --node[right] {\tiny{$\beta(0)$}} (1,1);
\draw[ForestGreen] (2,0) --node[below right] {\tiny{$\beta(1)$}} (2,0);
\draw[ForestGreen] (1,-1) --node[right] {\tiny{$\beta(2)$}} (1,-1);

\end{tikzpicture}$$

Then $\alpha*\beta$ is a length 3 looped walk $(v_1v_2v_3v_4)$.

$$
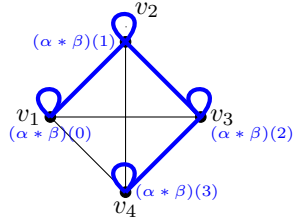
\begin{tikzpicture}% Image of alphas

\draw[fill] (0,0) circle (2pt);
\draw (0,0) --node[left] {$v_1$} (0,0);
\draw (0,0)  to[in=50,out=140,loop, distance=.7cm] (0,0);
\draw[fill] (1,1) circle (2pt);
\draw (1,1.2) --node[above right] {$v_2$} (1,1.2);
\draw (1,1)  to[in=50,out=140,loop, distance=.7cm] (1,1);
\draw[fill] (2,0) circle (2pt);
\draw (2,0)  to[in=50,out=140,loop, distance=.7cm] (2,0);
\draw (2,0) --node[right] {$v_3$} (2,0);
\draw[fill] (1,-1) circle (2pt);
\draw (1,-1)  to[in=50,out=140,loop, distance=.7cm] (1,-1);
\draw (1,-1) --node[below] {$v_4$} (1,-1);

\draw (0,0)--(1,1)--(2,0)--(1,-1)--(0,0);
\draw (0,0) -- (2,0);
\draw (1,1) -- (1,-1);

\draw[ultra thick, blue] (0,0)--(1,1)--(2,0)--(1,-1);

\draw[ultra thick, blue] (0,0)  to[in=50,out=140,loop, distance=.7cm] (0,0);
\draw[ultra thick, blue] (1,1)  to[in=50,out=140,loop, distance=.7cm] (1,1);
\draw[ultra thick, blue] (2,0)  to[in=50,out=140,loop, distance=.7cm] (2,0);
\draw[ultra thick, blue] (1,-1)  to[in=50,out=140,loop, distance=.7cm] (1,-1);

\draw[blue] (0,0) --node[below] {\tiny{$(\alpha*\beta)(0)$}} (0,0);
\draw[blue] (1,1) --node[left] {\tiny{$(\alpha*\beta)(1)$}} (1,1);
\draw[blue] (2,0) --node[below right] {\tiny{$(\alpha*\beta)(2)$}} (2,0);
\draw[blue] (1,-1) --node[right] {\tiny{$(\alpha*\beta)(3)$}} (1,-1);

\end{tikzpicture}$$

\end{example}

\begin{obs} \label{O:unit}  For any vertex $x$, there is a constant  length $0$ walk $c_x$ from $x$ to $x$ defined by $c_x(0)= x$.  Then for any other walk $\alpha$ from $x$ to $y$, $c_x * \alpha = \alpha$ and $\alpha * c_y = \alpha$.   If $x$ is looped, we can similarly define a constant  looped  walk at $x$.      \end{obs}

It is also straightforward to compare definitions and see both of the following:  

\begin{lemma}\label{L:assoc}  Contatenation of  [ordinary or looped] walks is associative:  when the endpoints match up to make  concatenation defined, we have $(\alpha * \beta) * \gamma = \alpha * (\beta * \gamma)$ 
\end{lemma}

\begin{lemma}\label{L:distr} Contatenation of [ordinary or looped] walks is distributive:  when $\phi$ and $\psi$ are graph homomorphisms, then  $\phi(g*h) = \phi g * \phi h$ and $(g * h) \psi = g \psi * h \psi$.  
\end{lemma}

\section{Graphs as a 2-Category}\label{S:HT} 

 In this section, we show that $\Gph$ has the structure of a 2-category as defined in \cite{riehlCTIC} with $\times$-homotopies between morphisms as 2-cells.

We define homotopy between graph morphisms $G \to H$ via the graph $G \times I^{\ell}_n$.  Because we use a looped interval graph,  we have a graph inclusion $G \cong G \times \{ k\}  \hookrightarrow G \times I^{\ell}_n$ for each vertex $k$ of $I_n^{\ell}$.

\begin{definition}\cite{Docht1} \label{D:htpy} Given $f, g:  G \to H$, we say that $f$ is {\bf $\times$-homotopic} to $g$, written $f \simeq g$,  if there is a map $\Lambda: G \times I_n^{\ell}  \to H$ such that $\Lambda | _{G \times \{ 0\} } = f$ and $\Lambda | _{G \times \{ n\} } = g$.   We will say $\Lambda$ is a length $n$ homotopy.        \end{definition}

This is defined as  $\times$-homotopy  in \cite{Docht1} to distinguish it from other graph homotopy notions considered in that paper, such as $A$-homotopy, and the usual notion of homotopy of spaces or simplicial complexes.  
Since this is the only version of homotopy that we will consider in this paper, we will also refer to it simply as 'homotopy'.

\begin{obs}\cite{Docht1}\label{O:seq}   By Propostion \ref{P:adj}, a morphism $ \Lambda: G \times I_n^{\ell}  \to H$ is equivalent to a morphism $\Lambda:  I_n^{\ell} \to H^G$.   Since all the vertices of $I_n^{\ell} $ are looped, they can only be mapped to looped vertices in $H^G$ which correspond to graph morphisms by Lemma \ref{P:adj}.    So the restriction of $H$ to $G \times \{k\}$ always gives a graph morphism, and a length $n$ $\times$-homotopy corresponds to a sequence of graph morphisms $(ff_1 f_2 f_3 \dots f_{n-1} g)$ such that $f_i \con f_{i+1} \in E(H^G)$.   Thus we can think of a $\times$-homotopy from $f$ to $g$ as a  looped walk in the exponential object $H^G$.  We will switch between these two views of homotopy as convenient.

\end{obs}

\begin{obs}\cite{Docht1} \label{O:htpyequiv}    $f\simeq g$ defines an equivalence relation on morphisms $G \to H$.

\end{obs}

    \begin{example}\label{E:IntimesG}  
 Suppose we have the graph $ \begin{tikzpicture} 
\node at (-1,0){$G=P_2$};

\draw[fill] (0,0) circle (2pt);
\draw (0,0) --node[below]{$a$} (0,0);

\draw[fill] (1,0) circle (2pt);
\draw (1,0) --node[below]{$b$} (1,0);

\draw[fill] (2,0) circle (2pt);
\draw (2,0) --node[below]{$c$} (2,0);

\draw (0,0) -- (1,0)--(2,0);
\end{tikzpicture}$

Consider the maps $\id_G, f:G\to G$ where $f(a)=f(c)=a$ and $f(b)=b$.    We abbreviate these morphisms by listing the images of vertices $a, b,$ and $c$ in order, so $id_G = abc$ and $f = aba$.

$$
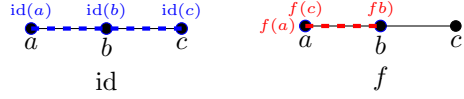
\begin{tikzpicture}

\draw[fill] (3,0) circle (2pt);
\draw[blue] (3,0) circle (2.2pt);
\draw (3,0) --node[below] {$a$} (3,0);
\draw (3,0) --node[above, blue] {\tiny{$\id(a)$}} (3,0);
\draw[fill] (4,0) circle (2pt);
\draw[blue] (4,0) circle (2.2pt);
\draw (4,0) --node[below] {$b$} (4,0);
\draw (4,0) --node[above, blue] {\tiny{$\id(b)$}} (4,0);
\draw[fill] (5,0) circle (2pt);
\draw[blue] (5,0) circle (2.2pt);
\draw (5,0) --node[below] {$c$} (5,0);
\draw (5,0) --node[above, blue] {\tiny{$\id(c)$}} (5,0);

\draw (3,0)--(4,0)--(5,0);
\draw[blue, ultra thick, dashed] (3,0)--(4,0)--(5,0);

\draw (4, -.7) -- node{$\id$} (4,-.7);

\end{tikzpicture}
\phantom{ww}
\begin{tikzpicture}

\draw[fill] (3,0) circle (2pt);
\draw[blue] (3,0) circle (2.2pt);
\draw (3,0) --node[below] {$a$} (3,0);
\draw (3,0) --node[left, red] {\tiny{$f(a)$}} (3,0);
\draw (3,0) --node[above, red] {\tiny{$f(c)$}} (3,0);
\draw[fill] (4,0) circle (2pt);
\draw[blue] (4,0) circle (2.2pt);
\draw (4,0) --node[below] {$b$} (4,0);
\draw (4,0) --node[above, red] {\tiny{$fb)$}} (4,0);
\draw[fill] (5,0) circle (2pt);
\draw (5,0) --node[below] {$c$} (5,0);

\draw (3,0)--(4,0)--(5,0);
\draw[red, ultra thick, dashed] (3,0)--(4,0);

\draw (4, -.7) -- node{$f$} (4,-.7);

\end{tikzpicture}
$$

We can define a homotopy $\Lambda:G\times I_1^{\ell}\to G$ from $id_G$ to $f$, where $\Lambda((x,0))=x$ and $\Lambda((x,1))=f(x)$.  Since $0, 1$ are both looped in $I_1^{\ell}$, the subgraphs $ G \times \{ 0 \}$ and $G \times \{1\}$ are both isomorphic to $G$.  It is easy to verify that $\Lambda$ is a graph homomorphism and thus is a length 1 homotopy.

 $$
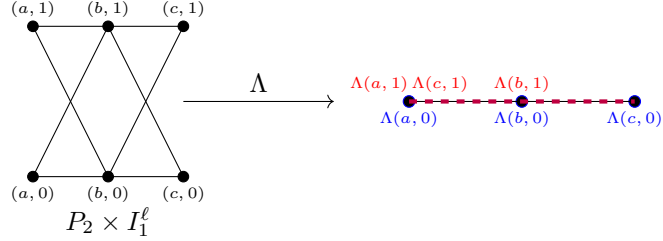
\begin{tikzpicture}

\draw[fill] (0,0) circle (2pt);
\draw (0,0) --node[below]{\tiny $(a,0)$} (0,0);

\draw[fill] (1,0) circle (2pt);
\draw (1,0) --node[below]{\tiny $(b,0)$} (1,0);

\draw[fill] (2,0) circle (2pt);
\draw (2,0) --node[below]{\tiny $(c,0)$} (2,0);

\draw[fill] (0,2) circle (2pt);
\draw (0,2) --node[above]{\tiny $(a,1)$} (0,2);

\draw[fill] (1,2) circle (2pt);
\draw (1,2) --node[above]{\tiny $(b,1)$} (1,2);

\draw[fill] (2,2) circle (2pt);
\draw (2,2) --node[above]{\tiny $(c,1)$} (2,2);

\draw(0,0) -- (1,2);
\draw(0,0) -- (1,0);
\draw(0,2) -- (1,0);
\draw(0,2) -- (1,2);
\draw(1,0) -- (2,2);
\draw(1,0) -- (2,0);
\draw(1,2) -- (2,0);
\draw(1,2) -- (2,2);

\node at (1,-.6){$P_2\times I_1^{\ell}$};

\draw[->](2, 1) -- node[above]{$\Lambda$} (4,1);

\draw[fill] (5,1) circle (2pt);
\draw[blue] (5,1) circle (2.2pt);
\draw (5,1) --node[below, blue] {\tiny$\Lambda(a,0)$} (5,1);
\draw (5,1) --node[above , red] {\tiny{$\Lambda(a,1)\, \Lambda(c,1)$}} (5,1);
\draw[fill] (6.5,1) circle (2pt);
\draw[blue] (6.5,1) circle (2.2pt);
\draw (6.5,1) --node[below, blue] {\tiny$\Lambda(b,0)$} (6.5,1);
\draw (6.5,1) --node[above, red] {\tiny{$\Lambda(b,1)$}} (6.5,1);
\draw[fill] (8,1) circle (2pt);
\draw[blue] (8,1) circle (2.2pt);
\draw (8,1) --node[below,blue] {\tiny $\Lambda(c,0)$} (8,1);

\draw (5,1)--(6.5,1)--(8,1);
\draw[purple, ultra thick, dashed] (5,1)--(6.5,1)--(8,1);

\end{tikzpicture}$$

 \end{example}

\begin{lemma}\label{L:comp3}
Suppose that $g\simeq g':  G \to H$.  If  $h:H\to K$, then $hg \simeq hg'$;  and if $f:  F \to G$, then   $ gf\simeq g'f$.
\end{lemma}
\begin{proof}
Since $g\simeq g'$, there is a length $n$ homotopy  $\Lambda$ from  $g$ to $g'$ in $H^G$.  Then $h_*\Lambda$ defines a  length $n$ homotopy from $hg $ to $hg'$  by Lemma \ref{L:comp1}.  Similarly, $\Lambda (f \times id_{I^{\ell}_{n}})^* $ defines a  length $n$ homotopy from $gf $ to $g'f$ by Lemma \ref{L:comp2}.  

\end{proof}

\begin{definition}[Concatenation of Homotopies]  \label{D:conH} Given $\Lambda_1:  f \simeq g$ and $\Lambda_2:  g \simeq h$, we define $\Lambda_1 * \Lambda _2:  f \simeq h$ using the concatenation of looped walks in $G^H$ of Definition \ref{D:concat}.   
\end{definition}

%\begin{obs}\label{O:length} that if $\Lambda_1$ is length $n_1$ and $\Lambda_2$ is length $n_2$, then $\Lambda_1*\Lambda_2$ is a homotopy of length $n_1 + n_2$.   \end{obs}

\begin{example}\label{E:examver}

Let $G=C_4$ and $H=P_2$ with vertices labeled as below.  

$$
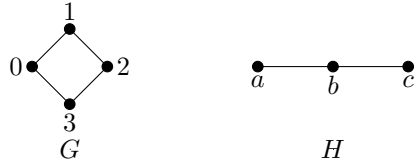
\begin{tikzpicture}

\draw[fill] (0,0) circle (2pt);
\draw (0,0) --node[left] {0} (0,0);
\draw[fill] (1/2,1/2) circle (2pt);
\draw (1/2,1/2) --node[above] {1} (.5,.5);
\draw[fill] (1,0) circle (2pt);
\draw (1,0) --node[right] {2} (1,0);
\draw[fill] (1/2,-1/2) circle (2pt);
\draw (1/2,-1/2) --node[below] {3} (.5,-.5);

\draw (0,0)--(.5,.5)--(1,0)--(.5,-.5)--(0,0);

\draw (.5, -1.1) -- node{$G$} (.5,-1.1);

\draw[fill] (3,0) circle (2pt);
\draw (3,0) --node[below] {$a$} (3,0);
\draw[fill] (4,0) circle (2pt);
\draw (4,0) --node[below] {$b$} (4,0);
\draw[fill] (5,0) circle (2pt);
\draw (5,0) --node[below] {$c$} (5,0);

\draw (3,0)--(4,0)--(5,0);

\draw (4, -1.1) -- node{$H$} (4,-1.1);

\end{tikzpicture}$$

Let $f:G\to H$ be defined by $f(0)=f(2)=b, f(1)=a, f(3)=c$.  Again, we will abbreviate this morphism by listing the images of $0, 1, 2, 3$ in order, so $f = babc$.  Let $f':G\to 
H$ be defined by $baba$, and let $f'':G\to H$ be defined by $bcbc$.  One can check that $f,f',f''\in \Gph(G, H)$.  

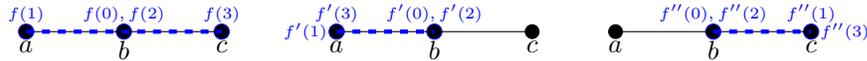
$$\begin{tikzpicture}[scale=1.3]

\draw[fill] (3,0) circle (2pt);
\draw[blue] (3,0) circle (2.2pt);
\draw (3,0) --node[below] {$a$} (3,0);
\draw (3,0) --node[above, blue] {\tiny{$f(1)$}} (3,0);
\draw[fill] (4,0) circle (2pt);
\draw[blue] (4,0) circle (2.2pt);
\draw (4,0) --node[below] {$b$} (4,0);
\draw (4,0) --node[above, blue] {\tiny{$f(0), f(2)$}} (4,0);
\draw[fill] (5,0) circle (2pt);
\draw[blue] (5,0) circle (2.2pt);
\draw (5,0) --node[below] {$c$} (5,0);
\draw (5,0) --node[above, blue] {\tiny{$f(3)$}} (5,0);

\draw (3,0)--(4,0)--(5,0);
\draw[blue, ultra thick, dashed] (3,0)--(4,0)--(5,0);

%\draw (4, -1.1) -- node{$\Img(f)$} (4,-1.1);

\end{tikzpicture}
\ \ \ 
\begin{tikzpicture}[scale=1.3]

\draw[fill] (3,0) circle (2pt);
\draw[blue] (3,0) circle (2.2pt);
\draw (3,0) --node[below] {$a$} (3,0);
\draw (3,0) --node[left, blue] {\tiny{$f'(1)$}} (3,0);
\draw (3,0) --node[above, blue] {\tiny{$f'(3)$}} (3,0);
\draw[fill] (4,0) circle (2pt);
\draw[blue] (4,0) circle (2.2pt);
\draw (4,0) --node[below] {$b$} (4,0);
\draw (4,0) --node[above, blue] {\tiny{$f'(0), f'(2)$}} (4,0);
\draw[fill] (5,0) circle (2pt);
\draw (5,0) --node[below] {$c$} (5,0);

\draw (3,0)--(4,0)--(5,0);
\draw[blue, ultra thick, dashed] (3,0)--(4,0);

%\draw (4, -1.1) -- node{$\Img(f')$} (4,-1.1);

\end{tikzpicture}
\ \ \ \  \ \ 
\begin{tikzpicture}[scale=1.3]

\draw[fill] (3,0) circle (2pt);
\draw (3,0) --node[below] {$a$} (3,0);
\draw[fill] (4,0) circle (2pt);
\draw[blue] (4,0) circle (2.2pt);
\draw (4,0) --node[below] {$b$} (4,0);
\draw (4,0) --node[above, blue] {\tiny{$f''(0), f''(2)$}} (4,0);
\draw[fill] (5,0) circle (2pt);
\draw[blue] (5,0) circle (2.2pt);
\draw (5,0) --node[below] {$c$} (5,0);
\draw (5,0) --node[above, blue] {\tiny{$f''(1)$}} (5,0);
\draw (5,0) --node[right, blue] {\tiny{$f''(3)$}} (5,0);

\draw (3,0)--(4,0)--(5,0);
\draw[blue, ultra thick, dashed] (4,0)--(5,0);

%\draw (4, -1.1) -- node{$\Img(f'')$} (4,-1.1);

\end{tikzpicture}
$$

Since $f \con f'\in E(H^G)$   we have a length $1$ homotopy 
 $\alpha:I_1^{\ell} \to H^G$ defined by $\alpha(0)=f, \alpha(1)=f'$.  Similarly, $f' \con f'' \in E(H^G)$  and so we have a homotopy  $\alpha':I_1^{\ell} \to H^G $ defined by $\alpha'(0)=f', \alpha'(1)=f''$.  Then $\alpha* \alpha':I_2^{\ell} \to H^G$ is defined by the looped walk $(f f' f'')$ in $H^G$, depicted in Figure 1 below.

\begin{figure}[h] 
$$\begin{tikzpicture}% Image of alphas

\draw[fill] (0,0) circle (2pt);
\draw (0,0) --node[left] {$babc$} (0,0);
\draw (0,0)  to[in=50,out=140,loop, distance=.7cm] (0,0);
\draw[fill] (1,1) circle (2pt);
\draw (1,1.2) --node[above] {$baba$} (1,1.2);
\draw (1,1)  to[in=50,out=140,loop, distance=.7cm] (1,1);
\draw[fill] (2,0) circle (2pt);
\draw (2,0)  to[in=50,out=140,loop, distance=.7cm] (2,0);
\draw (2,0) --node[right] {$bcbc$} (2,0);
\draw[fill] (1,-1) circle (2pt);
\draw (1,-1)  to[in=50,out=140,loop, distance=.7cm] (1,-1);
\draw (1,-1) --node[below] {$bcba$} (1,-1);

\draw (0,0)--(1,1)--(2,0)--(1,-1)--(0,0);
\draw (0,0) -- (2,0);
\draw (1,1) -- (1,-1);

\draw[dashed, ultra thick, red] (0,0)--(1,1);
\draw[dotted, ultra thick, ForestGreen] (1,1)--(2,0);

\draw[dashed, ultra thick, red] (0,0)  to[in=50,out=140,loop, distance=.7cm] (0,0);
\draw[dashed, ultra thick, red] (1,1)  to[in=50,out=140,loop, distance=.7cm] (1,1);
\draw[dotted, ultra thick, ForestGreen] (1,1)  to[in=50,out=140,loop, distance=.7cm] (1,1);

\draw[dotted, ultra thick, ForestGreen] (2,0)  to[in=50,out=140,loop, distance=.7cm] (2,0);

\draw[red] (0,0) --node[below] {\tiny{$\alpha(0)$}} (0,0);
\draw[red] (1,1) --node[left] {\tiny{$\alpha(1)$}} (1,1);

\draw[ForestGreen] (1,1) --node[right] {\tiny{$\alpha'(0)$}} (1,1);
\draw[ForestGreen] (2,0) --node[below] {\tiny{$\alpha'(1)$}} (2,0);

\draw[fill] (4,0) circle (2pt);
\draw (4,0) --node[left] {$cbab$} (4,0);
\draw (4,0)  to[in=50,out=140,loop, distance=.7cm] (4,0);
\draw[fill] (5,1) circle (2pt);
\draw (5,1.2) --node[above] {$abab$} (5,1.2);
\draw (5,1)  to[in=50,out=140,loop, distance=.7cm] (5,1);
\draw[fill] (6,0) circle (2pt);
\draw (6,0)  to[in=50,out=140,loop, distance=.7cm] (6,0);
\draw (6,0) --node[right] {$cbcb$} (6,0);
\draw[fill] (5,-1) circle (2pt);
\draw (5,-1)  to[in=50,out=140,loop, distance=.7cm] (5,-1);
\draw (5,-1) --node[below] {$abcb$} (5,-1);

\draw (4,0)--(5,1)--(6,0)--(5,-1)--(4,0);
\draw (4,0) -- (6,0);
\draw (5,1) -- (5,-1);

\end{tikzpicture}$$
\caption{Here we have depicted the walks $\alpha$ and $\alpha'$.   We have drawn only the subgraph of $H^G$ induced by graph homomorphisms rather than the whole exponential graph.   }
\end{figure}
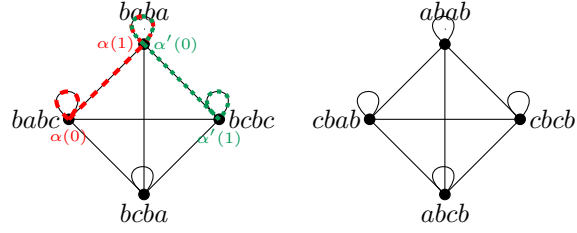

\end{example}

\begin{proposition}\label{P:ua1} The concatenation operation on homotopies is unital and associative. \end{proposition}

\begin{proof}  The constant homotopy defines a unit by Observation \ref{O:unit}, and associativity is given by Lemma \ref{L:assoc}. 

\end{proof}

We now define another composition of homotopies.  

\begin{definition}[Composition of Homotopies]  \label{D:conH2} Suppose that $f, f':  G\to H$ and $g, g':  H \to K$.    Given $\alpha:  f \simeq f'$ and $\beta:  g \simeq g'$, we define  $\alpha \circ \beta$ from $gf$ to $g'f'$ as follows:  let $g \alpha = g_* \alpha$ denote the homotopy from $gf$ to $gf'$,   and $(f')^* \beta = \beta f'$ denote the homotopy from $gf'$ to $g'f'$, as defined in Lemma \ref{L:comp3}.  Then  $$\alpha \circ \beta  = g \alpha * \beta f'.$$ 
\end{definition}

\begin{example}\label{examhor}
 As in Example \ref{E:examver}, let $G=C_4, H=K=P_2$ and    let $f:G\to H$ be defined by $babc$, and  $f':G\to H$ by $baba$, with $\alpha$ the  length 1 homotopy $(ff')$.

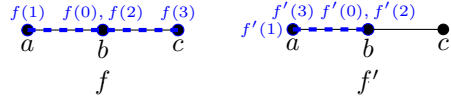
$$\begin{tikzpicture}

\draw[fill] (3,0) circle (2pt);
\draw[blue] (3,0) circle (2.2pt);
\draw (3,0) --node[below] {$a$} (3,0);
\draw (3,0) --node[above, blue] {\tiny{$f(1)$}} (3,0);
\draw[fill] (4,0) circle (2pt);
\draw[blue] (4,0) circle (2.2pt);
\draw (4,0) --node[below] {$b$} (4,0);
\draw (4,0) --node[above, blue] {\tiny{$f(0), f(2)$}} (4,0);
\draw[fill] (5,0) circle (2pt);
\draw[blue] (5,0) circle (2.2pt);
\draw (5,0) --node[below] {$c$} (5,0);
\draw (5,0) --node[above, blue] {\tiny{$f(3)$}} (5,0);

\draw (3,0)--(4,0)--(5,0);
\draw[blue, ultra thick, dashed] (3,0)--(4,0)--(5,0);

\draw (4, -.7) -- node{$f$} (4,-.7);

\end{tikzpicture}
\ \ \ 
\begin{tikzpicture}

\draw[fill] (3,0) circle (2pt);
\draw[blue] (3,0) circle (2.2pt);
\draw (3,0) --node[below] {$a$} (3,0);
\draw (3,0) --node[left, blue] {\tiny{$f'(1)$}} (3,0);
\draw (3,0) --node[above, blue] {\tiny{$f'(3)$}} (3,0);
\draw[fill] (4,0) circle (2pt);
\draw[blue] (4,0) circle (2.2pt);
\draw (4,0) --node[below] {$b$} (4,0);
\draw (4,0) --node[above, blue] {\tiny{$f'(0), f'(2)$}} (4,0);
\draw[fill] (5,0) circle (2pt);
\draw (5,0) --node[below] {$c$} (5,0);

\draw (3,0)--(4,0)--(5,0);
\draw[blue, ultra thick, dashed] (3,0)--(4,0);

\draw (4, -.7) -- node{$f'$} (4,-.7);

\end{tikzpicture}
$$

Let $g:H\to K$ be defined by $g(a)=g(c)=b, g(b)=a$ and let $g':H\to K$ be defined by $g'(a)=g'(c)=b, g'(b)=c$, with $\beta$ the length 1 homotopy $(gg')$.

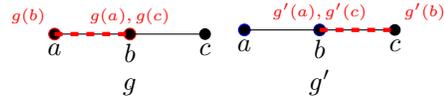
$$\begin{tikzpicture}

\draw[fill] (3,0) circle (2pt);
\draw[red] (3,0) circle (2.2pt);
\draw (3,0) --node[below] {$a$} (3,0);
\draw (3,0) --node[above left, red] {\tiny{$g(b)$}} (3,0);
\draw[fill] (4,0) circle (2pt);
\draw[red] (4,0) circle (2.2pt);
\draw (4,0) --node[below] {$b$} (4,0);
\draw (4,0) --node[above, red] {\tiny{$g(a), g(c)$}} (4,0);
\draw[fill] (5,0) circle (2pt);
\draw (5,0) --node[below] {$c$} (5,0);

\draw (3,0)--(4,0)--(5,0);
\draw[red, ultra thick, dashed] (3,0)--(4,0);

\draw (4, -.7) -- node{$g$} (4,-.7);

\end{tikzpicture}
\ 
\begin{tikzpicture}

\draw[fill] (3,0) circle (2pt);
\draw[blue] (3,0) circle (2.2pt);
\draw (3,0) --node[below] {$a$} (3,0);

\draw[fill] (4,0) circle (2pt);
\draw[blue] (4,0) circle (2.2pt);
\draw (4,0) --node[below] {$b$} (4,0);
\draw (4,0) --node[above, red] {\tiny{$g'(a), g'(c)$}} (4,0);
\draw[fill] (5,0) circle (2pt);
\draw (5,0) --node[below] {$c$} (5,0);
\draw (5,0) --node[above right, red] {\tiny{$g'(b)$}} (5,0);

\draw (3,0)--(4,0)--(5,0);
\draw[red, ultra thick, dashed] (4,0)--(5,0);

\draw (4, -.7) -- node{$g'$} (4,-.7);

\end{tikzpicture}
$$

Then $\alpha \circ \beta$ is a length $2$ homotopy $\beta_ \alpha:I_2^{\ell}\to K^G$  defined by the looped walk $(g f \,\, gf' \,\, g'f')$.   Concretely, both $gf$ and $gf'$ are given by the map $abab$ and $g'f'$ is defined by $cbcb$.  Thus  $\alpha \circ \beta$ is a length $2$ homotopy defined by the walk $(abab \,\, abab \,\, cbcb)$.  

$$
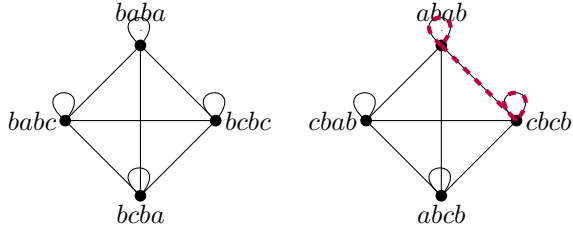
\begin{tikzpicture}% beta\alpha

\draw[fill] (0,0) circle (2pt);
\draw (0,0) --node[left] {$babc$} (0,0);
\draw (0,0)  to[in=50,out=140,loop, distance=.7cm] (0,0);
\draw[fill] (1,1) circle (2pt);
\draw (1,1.2) --node[above] {$baba$} (1,1.2);
\draw (1,1)  to[in=50,out=140,loop, distance=.7cm] (1,1);
\draw[fill] (2,0) circle (2pt);
\draw (2,0)  to[in=50,out=140,loop, distance=.7cm] (2,0);
\draw (2,0) --node[right] {$bcbc$} (2,0);
\draw[fill] (1,-1) circle (2pt);
\draw (1,-1)  to[in=50,out=140,loop, distance=.7cm] (1,-1);
\draw (1,-1) --node[below] {$bcba$} (1,-1);

\draw (0,0)--(1,1)--(2,0)--(1,-1)--(0,0);
\draw (0,0) -- (2,0);
\draw (1,1) -- (1,-1);

\draw[fill] (4,0) circle (2pt);
\draw (4,0) --node[left] {$cbab$} (4,0);
\draw (4,0)  to[in=50,out=140,loop, distance=.7cm] (4,0);
\draw[fill] (5,1) circle (2pt);
\draw (5,1.2) --node[above] {$abab$} (5,1.2);
\draw (5,1)  to[in=50,out=140,loop, distance=.7cm] (5,1);
\draw[fill] (6,0) circle (2pt);
\draw (6,0)  to[in=50,out=140,loop, distance=.7cm] (6,0);
\draw (6,0) --node[right] {$cbcb$} (6,0);
\draw[fill] (5,-1) circle (2pt);
\draw (5,-1)  to[in=50,out=140,loop, distance=.7cm] (5,-1);
\draw (5,-1) --node[below] {$abcb$} (5,-1);

\draw (4,0)--(5,1)--(6,0)--(5,-1)--(4,0);
\draw (4,0) -- (6,0);
\draw (5,1) -- (5,-1);

\draw[purple, dashed, ultra thick] (5,1)  to[in=50,out=140,loop, distance=.7cm] (5,1);
\draw[purple, dashed, ultra thick] (6,0)  to[in=50,out=140,loop, distance=.7cm] (6,0);
%\draw[purple] (5,1) --node[left] {\tiny{$(\beta_0\alpha)(0)$}} (5,1);
%\draw[purple] (5,1) --node[right] {\tiny{$(\beta\alpha)(1)$}} (5,1);
%\draw[purple] (6,0) --node[below right] {\tiny{$(\beta_0\alpha)(2)$}} (6,0);
\draw[purple, dashed, ultra thick] (5,1) -- (6,0);

\end{tikzpicture}$$

\end{example}

  We could equally well have chosen to define the composition as $\beta f * g'\alpha$.    This is not the same homotopy;  however, we will show that the two resulting homotopies are themselves homotopic.  
  To make this notion precise, we observe that a $\times$-homotopy $\alpha$ from $f$ to $g$ is defined as a looped walk  in $(H^G)^{I_n^{\ell}}$ given by $(f h_1 h_2 h_3 \dots h_{n-1} g)$.       Then for [looped or unlooped] walks, we define the notion of homotopy rel endpoints.   The idea of fixing a subspace and allowing only homotopies which are constant on this subspace is a common one from homotopy theory, and when the fixed subspace is $A$, this is referred to as homotopy rel $A$ \cite{Hatcher}.  In our case, we will take the subspace to be the end vertices of the path graph ${v_0, v_n}$.

Let $G$ be any graph.    Recall that a looped vertex of the exponential object  $G^{P_n}$ represents a length $n$ walk in $G$, and similarly a looped vertex of $G^{I_n^{\ell}}$ represents a looped walk in $G$.   Such an   $\alpha$ is given by  $(\alpha(0) \alpha(1) \alpha(2) \dots, \alpha(n) ) = (v_0 v_1 v_2 \dots, v_n)$.  
Define $s, t:  X^{P_n} \to X$ by $s(v_0 v_1 \dots v_n) = v_0$ and $t(v_0v_1 \dots v_n ) = v_{n}$.   Note that these are NOT graph homomorphisms, just maps of vertex sets.  Thus $\alpha$ is a walk from $x$ to $y$ if $s(\alpha) = x$ and $t(\alpha) = y$.

 \begin{definition}  \label{D:horelendpt} Suppose that $\alpha, \beta $ are walks in $G$  from $x$ to $y$.   We say $\alpha$ and $\beta$ are {\bf homotopic rel endpoints}   if they are homotopic in the subgraph 
 $$(G^{P_n})_{x, y} = \{\gamma \in G^{P_n} \, | \, s(\gamma) =x \textup{ and } t(\gamma) = y \}$$ Thus two walks $\alpha = (xv_1 \dots v_{n-1}y) $ and $\beta = (xw_1 \dots w_{n-1}y)$ are homotopic rel endpoints if there is a looped walk of walks in $G^{P_n}$ given by  $\Lambda = (\alpha \lambda_1\lambda_2 \dots \lambda_{k-1} \beta) $  where each walk $\lambda_i$ starts at $x$ and ends at $y$.\end{definition}

For looped walks, we make the same definitions in $G^{I_n^{\ell}}$.  

Now we apply this notion to $\times$-homotopies, viewed as looped walks in $(H^G)^{I_n^{\ell}}$.

\begin{definition}  \label{D:hoofho} Two $\times$-homotopies $\alpha, \alpha'$ from $f$ to $g$ are themselves homotopic if they are homotopic rel endpoints viewed as looped walks in  $(H^G)^{I_n^{\ell}}$.

\end{definition}
  
  \begin{proposition}\label{P:interchange} Suppose that $f, f':  G\to H$ and $g, g':  H \to K$.    Given $\alpha:  f \simeq f'$ and $\beta:  g \simeq g'$,  the two homotopies  defined by $g \alpha * \beta f'$ and $\beta f * g' \alpha$ are homotopic. \end{proposition}
  
  \begin{proof}   First, suppose that both $\alpha$ and $\beta $ are length $1$ homotopies, so that there are edges $f \con f'$ and $g \con g'$.  We consider the two length $2$ homotopies $g \alpha* \beta f' = (gf \, gf' \,  g'f')$, and $ \beta f * g' \alpha = (gf \, g'f \, g'f')$.    We want to show that these are homotopic.  In fact, we claim that they are connected by an edge in $K^G$.  Since $I^{\ell}_2$ has edges connecting $ 0\con 1$ and $1 \con 2$,   this requires that $(g \alpha * \beta f')(i) \con (\beta f * g'\alpha )({i+1}) $ and $(g \alpha * \beta f')(i+1) \con (\beta f * g'\alpha )({i}) $ for $i = 0, 1$.  So     there are four conditions to check.  Decoding them, they are:  $gf \con gf', gf \con g'f, g'f \con g'f'$ and $gf' \con g'f'$.  Each of these holds by   Lemma \ref{L:comp3}.  Lastly, we consider the loops $i \con i$:  for $i = 0, 2$ we have $\alpha(i) = \beta(i)$, and since these are looped vertices, $\alpha(i) \con \beta(i)$.  For $i = 1$, we have $\alpha(1) = g'f$ and $\beta(1) = gf'$.  If $v \con w \in E(G)$, then $f(v) \con f'(w) \in E(H)$ and hence $g'f(v) \con gf'(w) \in E(K) $, verifying the last condition.     Observe that this length $1$ homotopy fixes the endpoints, and thus we have a homotopy of homotopies (that is, the homotopies are homotopic rel endpoints).

$$
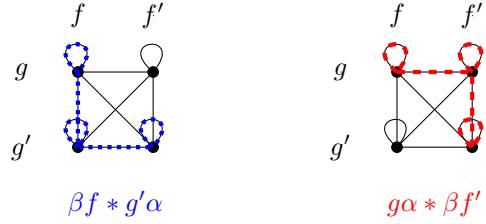
\begin{tikzpicture}
\draw[fill] (0,0) circle (2pt);
\draw[fill] (1,0) circle (2pt);
\draw[fill] (0,1) circle (2pt);
\draw[fill] (1,1) circle (2pt);
\draw (0,0)  to[in=50,out=140,loop, distance=.7cm] (0,0);
\draw (1,0)  to[in=50,out=140,loop, distance=.7cm] (1,0);
\draw (0,1)  to[in=50,out=140,loop, distance=.7cm] (0,1);
\draw (1,1)  to[in=50,out=140,loop, distance=.7cm] (1,1);

\draw (0,0) -- (0,1) -- (1,1) -- (0,0) -- (1,0) -- (1,1);
\draw (0,1) -- (1,0);

\draw (0,1.75) --node {$f$} (0,1.75);
\draw (1,1.75) --node {$f'$} (1,1.75);
\draw (-.75,1) --node {$g$} (-.75,1);
\draw (-.75,0) --node {$g'$} (-.75,0);

%\draw[red, ultra thick, dashed] (1,0)  to[in=50,out=140,loop, distance=.7cm] (1,0);
%\draw[red, ultra thick, dashed] (0,1)  to[in=50,out=140,loop, distance=.7cm] (0,1);
%\draw[red, ultra thick, dashed] (1,1)  to[in=50,out=140,loop, distance=.7cm] (1,1);
%\draw[red, ultra thick, dashed] (0,1) -- (1,1) -- (1,0);

\draw[blue, ultra thick, dotted] (0,0)  to[in=50,out=140,loop, distance=.7cm] (0,0);
\draw[blue, ultra thick, dotted] (1,0)  to[in=50,out=140,loop, distance=.7cm] (1,0);
\draw[blue, ultra thick, dotted] (0,1)  to[in=50,out=140,loop, distance=.7cm] (0,1);

\draw[blue, ultra thick, dotted] (0,1) -- (0,0) -- (1,0);

%\draw[red] (2,.5) --node {$g\alpha*\beta f'$} (2,.5);
\draw[blue] (.5,-.75) --node {$\beta f * g'\alpha$} (.5,-.75);

\end{tikzpicture}
\ \ \ \ \ \ \ \ \ \ \ \ \ \ \ 
\begin{tikzpicture}
\draw[fill] (0,0) circle (2pt);
\draw[fill] (1,0) circle (2pt);
\draw[fill] (0,1) circle (2pt);
\draw[fill] (1,1) circle (2pt);
\draw (0,0)  to[in=50,out=140,loop, distance=.7cm] (0,0);
\draw (1,0)  to[in=50,out=140,loop, distance=.7cm] (1,0);
\draw (0,1)  to[in=50,out=140,loop, distance=.7cm] (0,1);
\draw (1,1)  to[in=50,out=140,loop, distance=.7cm] (1,1);

\draw (0,0) -- (0,1) -- (1,1) -- (0,0) -- (1,0) -- (1,1);
\draw (0,1) -- (1,0);

\draw (0,1.75) --node {$f$} (0,1.75);
\draw (1,1.75) --node {$f'$} (1,1.75);
\draw (-.75,1) --node {$g$} (-.75,1);
\draw (-.75,0) --node {$g'$} (-.75,0);

\draw[red, ultra thick, dashed] (1,0)  to[in=50,out=140,loop, distance=.7cm] (1,0);
\draw[red, ultra thick, dashed] (0,1)  to[in=50,out=140,loop, distance=.7cm] (0,1);
\draw[red, ultra thick, dashed] (1,1)  to[in=50,out=140,loop, distance=.7cm] (1,1);
\draw[red, ultra thick, dashed] (0,1) -- (1,1) -- (1,0);

%\draw[blue, ultra thick, dotted] (0,0)  to[in=50,out=140,loop, distance=.7cm] (0,0);
%\draw[blue, ultra thick, dotted] (1,0)  to[in=50,out=140,loop, distance=.7cm] (1,0);
%\draw[blue, ultra thick, dotted] (0,1)  to[in=50,out=140,loop, distance=.7cm] (0,1);

%\draw[blue, ultra thick, dotted] (0,1) -- (0,0) -- (1,0);

\draw[red]  (.5,-.75) --node {$g\alpha*\beta f'$} (.5,-.75);
%\draw[white] (.5,-.75) --node {$\beta f * g'\alpha$} (.5,-.75);

\end{tikzpicture}
$$

  Now if $\alpha$ and $\beta$ are homotopies of length $n$ and $m$, each of them is defined by a looped walk $(f f_1  f_2  f_3  \dots  f_{n-1}  f')$ and $(g  g_1  g_2  g_3 \dots  g_{m-1}  g')$.    Since each successive pair is connected, the outer edges of each square are connected by an edge, ie a length 1 homotopy, and we can repeatedly swap squares and get a length $nm$ homotopy rel endpoints between    $g \alpha * \beta f'$ and $\beta f * g' \alpha$.

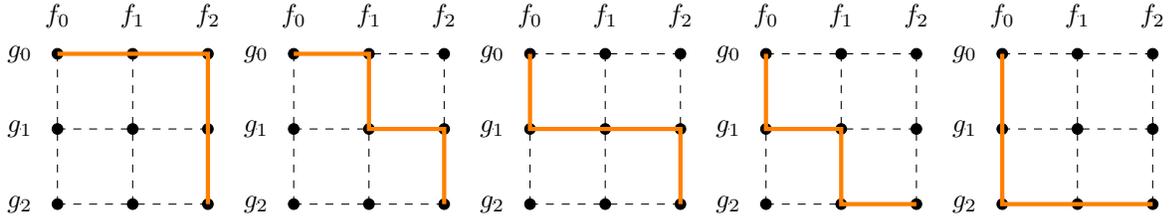
$$\begin{tikzpicture}
\draw[fill] (0,0) circle (2pt);
\draw[fill] (1,0) circle (2pt);
\draw[fill] (0,1) circle (2pt);
\draw[fill] (1,1) circle (2pt);
\draw[fill] (2,0) circle (2pt);
\draw[fill] (2,1) circle (2pt);
\draw[fill] (0,2) circle (2pt);
\draw[fill] (1,2) circle (2pt);
\draw[fill] (2,2) circle (2pt);

\draw[dashed] (0,0) -- (0,1) -- (0,2);
\draw[dashed] (1,0) -- (1,1) -- (1,2);
\draw[dashed] (2,0) -- (2,1) -- (2,2);

\draw[dashed] (0,0) -- (1,0) -- (2,0);
\draw[dashed] (0,1) -- (1,1) -- (2,1);
\draw[dashed] (0,2) -- (1,2) -- (2,2);

\draw(0,2.5) -- node{$f_0$} (0,2.5);
\draw(1,2.5) -- node{$f_1$} (1,2.5);
\draw(2,2.5) -- node{$f_2$} (2,2.5);

\draw(-.5,2) -- node{$g_0$} (-.5,2);
\draw(-.5,1) -- node{$g_1$} (-.5,1);
\draw(-.5,0) -- node{$g_2$} (-.5,0);

%\draw(.5,-.5) -- node{$\alpha$} (.5,-.5);
%\draw(1.5,-.5) -- node{$\alpha'$} (1.5,-.5);

%\draw(2.5,1.5) -- node{$\beta$} (2.5,1.5);
%\draw(2.5,.5) -- node{$\beta'$} (2.5,.5);

\draw[ultra thick, orange] (0,2)--(1,2)--(2,2)--(2,1)--(2,0);

%\draw[ForestGreen] (1,-1) -- node{$(\alpha'_1\alpha)_0(\beta'_1\beta)$} (1,-1);

\end{tikzpicture}
\ 
\begin{tikzpicture}
\draw[fill] (0,0) circle (2pt);
\draw[fill] (1,0) circle (2pt);
\draw[fill] (0,1) circle (2pt);
\draw[fill] (1,1) circle (2pt);
\draw[fill] (2,0) circle (2pt);
\draw[fill] (2,1) circle (2pt);
\draw[fill] (0,2) circle (2pt);
\draw[fill] (1,2) circle (2pt);
\draw[fill] (2,2) circle (2pt);

\draw[dashed] (0,0) -- (0,1) -- (0,2);
\draw[dashed] (1,0) -- (1,1) -- (1,2);
\draw[dashed] (2,0) -- (2,1) -- (2,2);

\draw[dashed] (0,0) -- (1,0) -- (2,0);
\draw[dashed] (0,1) -- (1,1) -- (2,1);
\draw[dashed] (0,2) -- (1,2) -- (2,2);

\draw(0,2.5) -- node{$f_0$} (0,2.5);
\draw(1,2.5) -- node{$f_1$} (1,2.5);
\draw(2,2.5) -- node{$f_2$} (2,2.5);

\draw(-.5,2) -- node{$g_0$} (-.5,2);
\draw(-.5,1) -- node{$g_1$} (-.5,1);
\draw(-.5,0) -- node{$g_2$} (-.5,0);

%\draw(.5,-.5) -- node{$\alpha$} (.5,-.5);
%\draw(1.5,-.5) -- node{$\alpha'$} (1.5,-.5);

%\draw(2.5,1.5) -- node{$\beta$} (2.5,1.5);
%\draw(2.5,.5) -- node{$\beta'$} (2.5,.5);

\draw[ultra thick, orange] (0,2)--(1,2)--(1,1)--(2,1)--(2,0);

%\draw[orange] (1,-1) -- node{$(\beta'_0\alpha')_1(\beta_0\alpha)$} (1,-1);

\end{tikzpicture}
\ 
\begin{tikzpicture}
\draw[fill] (0,0) circle (2pt);
\draw[fill] (1,0) circle (2pt);
\draw[fill] (0,1) circle (2pt);
\draw[fill] (1,1) circle (2pt);
\draw[fill] (2,0) circle (2pt);
\draw[fill] (2,1) circle (2pt);
\draw[fill] (0,2) circle (2pt);
\draw[fill] (1,2) circle (2pt);
\draw[fill] (2,2) circle (2pt);

\draw[dashed] (0,0) -- (0,1) -- (0,2);
\draw[dashed] (1,0) -- (1,1) -- (1,2);
\draw[dashed] (2,0) -- (2,1) -- (2,2);

\draw[dashed] (0,0) -- (1,0) -- (2,0);
\draw[dashed] (0,1) -- (1,1) -- (2,1);
\draw[dashed] (0,2) -- (1,2) -- (2,2);

\draw(0,2.5) -- node{$f_0$} (0,2.5);
\draw(1,2.5) -- node{$f_1$} (1,2.5);
\draw(2,2.5) -- node{$f_2$} (2,2.5);

\draw(-.5,2) -- node{$g_0$} (-.5,2);
\draw(-.5,1) -- node{$g_1$} (-.5,1);
\draw(-.5,0) -- node{$g_2$} (-.5,0);

%\draw(.5,-.5) -- node{$\alpha$} (.5,-.5);
%\draw(1.5,-.5) -- node{$\alpha'$} (1.5,-.5);

%\draw(2.5,1.5) -- node{$\beta$} (2.5,1.5);
%\draw(2.5,.5) -- node{$\beta'$} (2.5,.5);

\draw[ultra thick, orange] (0,2)--(0,1)--(1,1)--(2,1)--(2,0);

%\draw[orange] (1,-1) -- node{$(\beta'_0\alpha')_1(\beta_0\alpha)$} (1,-1);

\end{tikzpicture}
\ 
\begin{tikzpicture}
\draw[fill] (0,0) circle (2pt);
\draw[fill] (1,0) circle (2pt);
\draw[fill] (0,1) circle (2pt);
\draw[fill] (1,1) circle (2pt);
\draw[fill] (2,0) circle (2pt);
\draw[fill] (2,1) circle (2pt);
\draw[fill] (0,2) circle (2pt);
\draw[fill] (1,2) circle (2pt);
\draw[fill] (2,2) circle (2pt);

\draw[dashed] (0,0) -- (0,1) -- (0,2);
\draw[dashed] (1,0) -- (1,1) -- (1,2);
\draw[dashed] (2,0) -- (2,1) -- (2,2);

\draw[dashed] (0,0) -- (1,0) -- (2,0);
\draw[dashed] (0,1) -- (1,1) -- (2,1);
\draw[dashed] (0,2) -- (1,2) -- (2,2);

\draw(0,2.5) -- node{$f_0$} (0,2.5);
\draw(1,2.5) -- node{$f_1$} (1,2.5);
\draw(2,2.5) -- node{$f_2$} (2,2.5);

\draw(-.5,2) -- node{$g_0$} (-.5,2);
\draw(-.5,1) -- node{$g_1$} (-.5,1);
\draw(-.5,0) -- node{$g_2$} (-.5,0);

%\draw(.5,-.5) -- node{$\alpha$} (.5,-.5);
%\draw(1.5,-.5) -- node{$\alpha'$} (1.5,-.5);

%\draw(2.5,1.5) -- node{$\beta$} (2.5,1.5);
%\draw(2.5,.5) -- node{$\beta'$} (2.5,.5);

\draw[ultra thick, orange] (0,2)--(0,1)--(1,1)--(1,0)--(2,0);

%\draw[orange] (1,-1) -- node{$(\beta'_0\alpha')_1(\beta_0\alpha)$} (1,-1);

\end{tikzpicture}
\ 
\begin{tikzpicture}
\draw[fill] (0,0) circle (2pt);
\draw[fill] (1,0) circle (2pt);
\draw[fill] (0,1) circle (2pt);
\draw[fill] (1,1) circle (2pt);
\draw[fill] (2,0) circle (2pt);
\draw[fill] (2,1) circle (2pt);
\draw[fill] (0,2) circle (2pt);
\draw[fill] (1,2) circle (2pt);
\draw[fill] (2,2) circle (2pt);

\draw[dashed] (0,0) -- (0,1) -- (0,2);
\draw[dashed] (1,0) -- (1,1) -- (1,2);
\draw[dashed] (2,0) -- (2,1) -- (2,2);

\draw[dashed] (0,0) -- (1,0) -- (2,0);
\draw[dashed] (0,1) -- (1,1) -- (2,1);
\draw[dashed] (0,2) -- (1,2) -- (2,2);

\draw(0,2.5) -- node{$f_0$} (0,2.5);
\draw(1,2.5) -- node{$f_1$} (1,2.5);
\draw(2,2.5) -- node{$f_2$} (2,2.5);

\draw(-.5,2) -- node{$g_0$} (-.5,2);
\draw(-.5,1) -- node{$g_1$} (-.5,1);
\draw(-.5,0) -- node{$g_2$} (-.5,0);

%\draw(.5,-.5) -- node{$\alpha$} (.5,-.5);
%\draw(1.5,-.5) -- node{$\alpha'$} (1.5,-.5);

%\draw(2.5,1.5) -- node{$\beta$} (2.5,1.5);
%\draw(2.5,.5) -- node{$\beta'$} (2.5,.5);

\draw[ultra thick, orange] (0,2)--(0,1)--(0,0)--(1,0)--(2,0);

%\draw[orange] (1,-1) -- node{$(\beta'_0\alpha')_1(\beta_0\alpha)$} (1,-1);

\end{tikzpicture}$$

  \end{proof}

\begin{proposition}\label{P:ua2} The composition operation on homotopies is unital and associative. \end{proposition}

\begin{proof}  Unital:    If $\alpha$ is the constant homotopy at $f$, then  $g \alpha$ is just constant at $gf$, and $g \alpha * \beta f' = \beta f'$ by Observation \ref{O:unit}.  Similarly if $\beta$ is the constant homotopy at $gf'$, then $\beta f' = f'$ and $g \alpha * \beta f' = \beta f'$.  

Associative:    Suppose we have homotopies  $\alpha:  f \simeq f', \beta: g \simeq g'$ and $\gamma:  h \simeq h'$.    Then the distributive property of Lemma \ref{L:distr} and the associative property of  Lemma \ref{L:assoc}  give:  
\begin{align*}  (\alpha \circ \beta)  \circ \gamma & = ( g \alpha * \beta f') \circ \gamma \\ 
& =  h(g \alpha * \beta f') * \gamma g'f' \\
& = (hg  \alpha * h \beta f') * \gamma g'f'\\
& =hg  \alpha *(  h \beta f' * \gamma g'f' ) \\ 
& = hg \alpha *  (h \beta *  \gamma g')f' \\
& =  \alpha \circ  (h \beta * \gamma g') \\ & = \alpha \circ  (\beta \circ   \gamma)
\end{align*}  
\end{proof}

 Recall that our aim is to  show that $\Gph$ forms a 2-category as defined in \cite{riehlCTIC}.   We want our 2-cells to be defined by $\times$-homotopies of morphisms, but this does not  satisfy the required properties.   However, since a homotopy $\alpha$ is defined by a looped walk given by a  map  $\alpha:  I_n^{\ell} \to H^G$, we have a notion of when two such maps are themselves homotopic, as in Definiton  \ref{D:hoofho}.  In  order to get a  2-category, we will define our 2-cells to be homotopy classes of $\times$-homotopies.  

  We begin by showing that  concatenation and composition operations are well defined up to homotopy.
We will  use the following more general result about homotopies of walks:  

\begin{lemma}\label{L:conhtpy}  If $f$ and $g$ are looped walks of length $n$ in $G$ from $x$ to $y$, and $f \simeq g$ are homotopic rel endpoints, then if $h$ is a walk from $y$ to $z$, then $f*h \simeq g*h$ rel endpoints; and if $k$ is a walk from $w$ to $x$, then  $k*f \simeq k*g$ rel endpoints.   
\end{lemma}

\begin{proof}   We have $f$ and $g$ representing vertices in $G^{I_n^{\ell}}$, and $\alpha$  a length $m$ homotopy from $f$ to $g$. So   $\alpha$ is defined by a looped walk  $(f  f_1 f_2 \dots f_{n-1}  g)$  in $G^{I_n^{\ell}}$.   
Now suppose that  $h$ is a walk from $y$ to $z$.  Define a sequence $(f_0*h \,\, f_1*h \,\, f_2*h \dots f_n*h)$ in $G^{I_{n+m}^{\ell}}$, where each of these is a walk from $x$ to $z$.   We claim that each successive pair of these is connected by an edge  in $G^{I_{n+m}^{\ell}}$.  
The requirement for this edge to exist is that given any edge $v_i \con v_{i+1}$ in $I_{n+m}^{\ell}$,    we have $(f_k*h)(v_i) \con (f_{k+1}*h)(v_{i+1})$.    By definition of concatenation, if $i \leq n$ these are defined by $f_k(v_i)$ and $f_{k+1}(v_{i+1})$, which are connected in $G$ since $f_k \con f_{k+1}$;  if $i > n$, these are defined by  $h(v_i) $ and $h(v_{i+1})$, which are connected since $h$ is a walk in $G$.  
Thus $f*h \simeq g*h$ rel endpoints.  The other case follows by an analogous argument.  
\end{proof}

\begin{corollary}  \label{C:htht1}   If $\alpha \simeq \alpha'$ are homotopic as homotopies (ie homotopic rel endpoints) and $\beta \simeq \beta'$ as homotopies, then $\alpha * \beta \simeq \alpha' * \beta'$. 
\end{corollary}

\begin{lemma}\label{L:htht2}  If $\alpha \simeq \alpha'$ and $\beta \simeq \beta'$ then $\alpha \circ \beta \simeq \alpha' \circ  \beta'$. 
\end{lemma}

\begin{proof}  Start with $\alpha  \circ \beta = g\alpha * \beta f'$.  Now by Lemma \ref{L:comp3}, we have a homotopy $g\alpha \simeq g \alpha '$, and hence by Lemma \ref{L:conhtpy} a homotopy $g \alpha * \beta f' \simeq g \alpha' * \beta f'$.  Then  Lemma \ref{L:conhtpy} also says that $\beta f' \simeq \beta' f'$, and so $ g\alpha' * \beta f' \simeq g \alpha ' * \beta' f' = \alpha'  \beta'$.  Thus we have $\alpha  \circ \beta \simeq \alpha'  \circ \beta'$  as homotopies.  

\end{proof}

\begin{theorem}\label{thm2cat}
 We can define a 2-category  $\Gph$ as follows: 

\begin{itemize}
\item Objects [0-cells] are given by objects of $\Gph$, the finite undirected graphs. 
\item Arrows [1-cells] are given by the arrows of $\Gph$, the graph morphisms
\item Given  $f, f':  G \to H$, a 2-cell from $f$ to $f'$  is a homotopy rel endpoints class $[\alpha]$ of $\times$-homotopies $\alpha:  I_n^{\ell} \to H^G$ such that $\alpha:  f \simeq f'$.  
\item Vertical composition is defined using concatenation $[\alpha] * [\alpha'] = [\alpha * \alpha']$
\item Horizontal composition is defined using composition $[\alpha] \circ [\beta] = [\alpha  \circ \beta]$
\end{itemize}
  
 \end{theorem}
 
 \begin{proof}   We have shown that vertical and horizontal composition are well-defined in Corollary \ref{C:htht1} and Lemma \ref{L:htht2}, and that these operations are associative and unital in Propositions \ref{P:ua1} and \ref{P:ua2}.    Therefore what remains is to check the interchange law.

 Our set-up is as follows:  we have maps $f, f', f'':  G \to H$ and $g, g', g'':  H \to K$, with two cells $\alpha:  f \simeq f', \alpha':  f' \simeq f''$ and $\beta:  g\simeq  g', \beta': g' \simeq    g''$:  

 We want to show that $(\alpha \circ \beta) * (\alpha' \circ \beta') \simeq (\alpha * \alpha') \circ (\beta * \beta')$.    Unravelling the definitions here shows that $(\alpha \circ \beta) * (\alpha' \circ \beta') = (g \alpha* \beta f') * (g' \alpha' * \beta' f'')$, while $(\alpha * \alpha') \circ (\beta * \beta') = g (\alpha * \alpha') * (\beta * \beta') f'' = (g \alpha * g \alpha') * (\beta f'' * \beta' f'')$ using the distributivity of Lemma \ref{L:distr}.  Since concatenation is associative, we are comparing $g \alpha * \beta f' * g' \alpha' * \beta' f''$ with $g \alpha * g \alpha' * \beta f'' * \beta' f''$.    Therefore it suffices to show that $\beta f' * g'\alpha' \simeq g \alpha' * \beta f''$.   But this is exactly  Proposition \ref{P:interchange}.

 \end{proof}

%\begin{corollary}  If $f_1, g_1:G\to H$, s.t. $f_1\cong g_1$ and $f_2, g_2:H\to K$ s.t. $f_2\cong g_2$, then $g_1 f_1\cong g_2 f_2$.\end{corollary}

%\begin{proof}vv$g_1 f_1\cong g_1 f_2\cong g_2 f_2$.  \end{proof}

%%%%%%%%%%%%%%%%%%%%%%%%%%%%%%%%%%%%%%%%%%%%%%%%%%%%%%%%%%%%%%%%%%%%%%%

\section{Structure of Homotopies for Finite Graphs}\label{S:Struct}

In this section, 
we develop a more explicit description of  $\times$-homotopies between graph morphisms when $G$ is a finite graph.  
We show that such graph homotopies can always be defined 'locally', shifting one vertex at a time.      We imagine a spider walking through the graph by moving one leg at a time.  

\begin{definition} \label{D:spiderpair} Let $f, g:  G \to H$ be graph morphisms.  We say that $f$ and $g$ are a {\bf spider pair} if there is a single vertex of $G$, say $x$, such that $f(y) = g(y)$ for all $y \neq x$.  If $x$ is unlooped there are no additional conditions, but if $x\con x \in E(G)$, then we require  that $f(x) \con g(x) \in E(H)$.    When we replace $f$ with $g$ we refer to it as a {\bf spider move}.  

\end{definition}

\begin{lemma}\label{L:spiderhom} If $f$ and $g$ are a spider pair, then $f \con g  \in E(H^G)$.  
\end{lemma}

\begin{proof}  For any $y \con z \in E(G)$ we need to verify that that $f(y) \con g(z) \in E(H)$.  If $y, z \neq x$ then $g(z) = f(z)$ and so this follows from the fact that $f$ is a  graph morphism.  If $y \con x$ for $y \neq x$, then $f(y) \con g(x)$ since $f(y) = g(y)$  and $g$ is a graph morphism; similarly, $f(x) \con g(y)$.   Lastly, if $x \con x$, then we have asked that $f(x) \con g(x)$.  Therefore $f \con g$ have an edge in the exponential graph $H^G$.  

\end{proof}

\begin{example}\label{E:spiderpair}
Let $G$ and $H$ be the  graphs from Example \ref{E:examexp}:

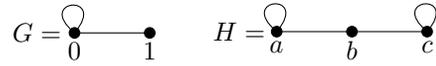
$$\begin{tikzpicture} 
\node at (-.5,0){$G=$};

\draw[fill] (0,0) circle (2pt);
\draw (0,0) --node[below]{0} (0,0);
\draw (0,0)  to[in=50,out=140,loop, distance=.7cm] (0,0);

\draw[fill] (1,0) circle (2pt);
\draw (1,0) --node[below]{1} (1,0);

\draw (0,0) -- (1,0);
\end{tikzpicture}\phantom{ww} 
\begin{tikzpicture} 
\node at (-.5,0){$H=$};

\draw[fill] (0,0) circle (2pt);
\draw (0,0) --node[below]{$a$} (0,0);
\draw (0,0)  to[in=50,out=140,loop, distance=.7cm] (0,0);

\draw[fill] (1,0) circle (2pt);
\draw (1,0) --node[below]{$b$} (1,0);

\draw[fill] (2,0) circle (2pt);
\draw (2,0) --node[below]{$c$} (2,0);
\draw (2,0)  to[in=50,out=140,loop, distance=.7cm] (2,0);

\draw (0,0) -- (1,0) -- (2,0);
\end{tikzpicture}.$$

Let $f, g:G\to H$ be defined by $f(0)=a, f(1)=b$, and  $g(0)=a, g(1)=a$.    So $f, g$ are a spider pair, and we  see that the morphisms $f,g$ are adjacent in the exponential object $H^G$.

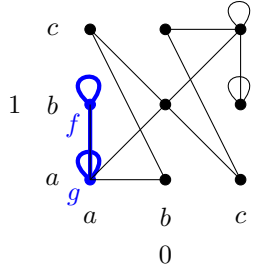
$$\begin{tikzpicture}

\node at (0,-.5){$a$};
\node at (1,-.5){$b$};
\node at (2,-.5){$c$};
\node at (1,-1){$0$};

\node at (-.5, 0){$a$};
\node at (-.5, 1){$b$};
\node at (-.5, 2){$c$};
\node at (-1, 1){$1$};

\draw[fill, blue] (0,0) circle (2pt);
\draw[fill] (1,0) circle (2pt);
\draw[fill] (2,0) circle (2pt);
\draw[fill, blue] (0,1) circle (2pt);
\draw[fill] (1,1) circle (2pt);
\draw[fill] (2,1) circle (2pt);
\draw[fill] (0,2) circle (2pt);
\draw[fill] (1,2) circle (2pt);
\draw[fill] (2,2) circle (2pt);

\draw[ultra thick, blue] (0,0)  to[in=50,out=140,loop, distance=.7cm] (0,0);
\draw[ultra thick, blue] (0,1)  to[in=50,out=140,loop, distance=.7cm] (0,1);
\draw (2,1)  to[in=50,out=140,loop, distance=.7cm] (2,1);
\draw (2,2)  to[in=50,out=140,loop, distance=.7cm] (2,2);

\draw[blue] (0,1) -- node[below left]{$f$} (0,1);
\draw[blue] (0,0) -- node[below left]{$g$} (0,0);

\draw[ultra thick, blue] (0,0) -- (0,1);
\draw (0,0) -- (1,0);
\draw (0,0) -- (1,1);
\draw (0,2) -- (1,0);
\draw (0,2) -- (1,1);
\draw (0,0) -- (0,1);
\draw (1,1) -- (2,0);
\draw (1,1) -- (2,2);
\draw (1,2) -- (2,0);
\draw (1,2) -- (2,2);
\draw (2,1) -- (2,2);

%\node at (1,2.5){$H^G=$};

\end{tikzpicture}$$

\end{example}

We now prove  that all homotopies with finite domain can be decomposed as a  sequences of spider moves, moving one vertex at a time.  

\begin{proposition}[Spider Lemma] \label{P:spider}  If $f, g:  G \to H$ and $G$ is a finite graph, and  $f \con g \in E(H^G)$, then there is a finite sequence of morphisms $f = f_0, f_1, f_2, \dots, f_n = g$ such that each successive pair $f_k, f_{k+1}$ is a spider pair.  

\end{proposition}

\begin{proof}
Since $G$ is a finite graph, we can label its vertices $v_1, v_2, \dots, v_n$.  Then for $0 \leq k \leq n$, we define:  
\begin{equation*}
        f_k(v_i) = \begin{cases}
                        f(v_i) &  \text{ for }  i \leq n-k \\
                        g(v_i) &  \text{ for }  i > n-k \\ 
                    \end{cases}
\end{equation*}

First we check that each  $f_k$ is a graph morphism.   Suppose $v_i \con v_j \in E(G)$;  we need to show that $f_k(v_i) \con f_k(v_j)$.   If $i, j \leq n-k$ then $f_k = f$ for both vertices, and so since $f$ is a morphism,  $f(v_i) \con f(v_j)$.  Similarly if $i, j > n-k$ then $f_k = g$ on both vertices.    Lastly, if $i\leq n-k$ and $j>n-k$, we know that $f \con g$ in $H^G$,  so by the structure of edges in the exponential object, $f(v_i ) \con g(v_j)$.  Thus $f_k(v_i) \con f_k(v_j)$.  

It is clear that each pair  $f_k, f_{k+1}$ agrees on every vertex except $v_{n-k}$.   So to show this is a spider pair,  we only need to check that if $v_{n-k}$ is looped, then $f_k(v_{n-k}) \con f_{k+1}(v_{n-k})$.    But since $f \con g \in E(H^G)$, we know that if $v_{n-k} \con v_{n-k}$ then  $f(v_{n-k}) \con g(v_{n-k})$.

\end{proof}

\begin{corollary}  Whenever $f, g:  G \to H$ with $G$ finite and $f \simeq g$, there is a finite sequence of spider moves connecting $f$ and $g$.  
\end{corollary}

Thus we can see  explicitly what $\times$-homotopies of graph morphisms betweem finite graphs can do.

\begin{example}
Let $G=C_4, H=P_2$ as in Example \ref{E:examver}.  The morphisms $f = babc$ and $g = bcba$ are adjacent in $H^G$.  They are not a spider pair since $f(1)\neq g(1)$ and $f(3)\neq g(3)$.  However, if we define $h = baba$; then there is a spider move $f$ to $h$, and another from $h$ to $g$, giving a sequence of spider moves from $f$ to $g$, shown in Figure 4.1.

\begin{figure}[h]\label{Figure:SpiderMoves}
$$\begin{tikzpicture}% Image of alphas

\draw[fill] (0,0) circle (2pt);
\draw (0,0) --node[left] {$babc$} (0,0);
\draw (0,0)  to[in=50,out=140,loop, distance=.7cm] (0,0);
\draw[fill] (1,1) circle (2pt);
\draw (1,1.2) --node[above] {$baba$} (1,1.2);
\draw (1,1)  to[in=50,out=140,loop, distance=.7cm] (1,1);
\draw[fill] (2,0) circle (2pt);
\draw (2,0)  to[in=50,out=140,loop, distance=.7cm] (2,0);
\draw (2,0) --node[right] {$bcbc$} (2,0);
\draw[fill] (1,-1) circle (2pt);
\draw (1,-1)  to[in=50,out=140,loop, distance=.7cm] (1,-1);
\draw (1,-1) --node[below] {$bcba$} (1,-1);

\draw (0,0)--(1,1)--(2,0)--(1,-1)--(0,0);
\draw (0,0) -- (2,0);
\draw (1,1) -- (1,-1);

\draw[ultra thick, blue] (0,0)--(1,-1);
\draw[ultra thick, red, dashed] (0,0)--(1,1)--(1,-1);

\draw[blue] (0,0) --node[below] {\tiny{$f$}} (0,0);
\draw[blue] (1,-1) --node[right] {\tiny{$g$}} (1,-1);
\draw[red] (1,1) --node[right] {\tiny{$h$}} (1,1);

\draw[fill] (4,0) circle (2pt);
\draw (4,0) --node[left] {$cbab$} (4,0);
\draw (4,0)  to[in=50,out=140,loop, distance=.7cm] (4,0);
\draw[fill] (5,1) circle (2pt);
\draw (5,1.2) --node[above] {$abab$} (5,1.2);
\draw (5,1)  to[in=50,out=140,loop, distance=.7cm] (5,1);
\draw[fill] (6,0) circle (2pt);
\draw (6,0)  to[in=50,out=140,loop, distance=.7cm] (6,0);
\draw (6,0) --node[right] {$cbcb$} (6,0);
\draw[fill] (5,-1) circle (2pt);
\draw (5,-1)  to[in=50,out=140,loop, distance=.7cm] (5,-1);
\draw (5,-1) --node[below] {$abcb$} (5,-1);

\draw (4,0)--(5,1)--(6,0)--(5,-1)--(4,0);
\draw (4,0) -- (6,0);
\draw (5,1) -- (5,-1);

\end{tikzpicture}$$

\caption{Decomposing a homotopy into spider moves.}

\end{figure}
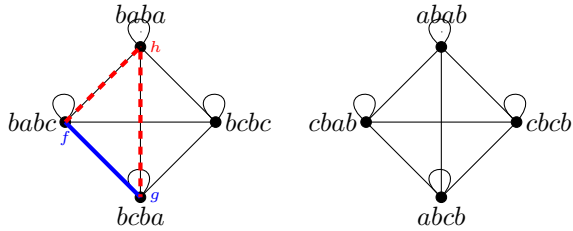
\end{example}

%%%%%%%%%%%%%%%%%%%%%%%%%%%%%%%%%%%%%%%%%%%%%%%%%%%%%%%%%%%%%%%%%%%%%%
 A special case of the spider moves can be used to analyze homotopy equivalences.   In the literature, homotopy has been linked to the idea of a {\bf fold} or a {\bf dismantling} \cite{HN2004, GMDG, Docht1, MoreFolding}.  This can be seen as a special case of our more general spider moves.

\begin{definition}\label{D:fold}  If $G$ is a graph, we say that a morphism $f:   G\to G$ is a {\bf fold} if $f$ and the identity map are a spider pair.  
\end{definition}

\begin{proposition} \label{P:fold}  If $f$ is a fold, then $f:  G \to Im(f)$  is a homotopy equivalence, where $Im(f)$ is as defined in Definition \ref{D:im}. 
\end{proposition}

\begin{proof}   Since $f$ and $id_G$ form a spider pair, the map $f$ is the identity on every vertex except one, call it $v$.  If $f(v) = v$ then $f$  is the identity and we are done.  

If $f(v)= w \neq v$, then $Im(f) = G\backslash\{ v\}$.   Consider $\iota:  Im(f) \to G$ to be the inclusion map.  Then the  composition $f\iota $ is the identity on $Im(f)$.  Now consider $\iota f:  G \to G$.  Since $\iota$ is just the inclusion of the image, $\iota f = f$.  By Lemma \ref{L:spiderhom}, $f \simeq \id$.  

\end{proof}

 We  identify when we have a potential fold by a condition on neighborhood of vertices.  In  \cite{GMDG} \cite{Docht1} folds are defined using  this condition.  We denote the neighbourhood of a vertex $v$ by $N(v) = \{ w \in V(G) \, | \, w \con v \}$.
 
\begin{proposition}\label{P:fold2}   Suppose that $f:  V(G) \to V(G)$ is a set map of vertices such that  $f$ is the identity on all vertices except one.  Explicitly there exists a vertex $w \in V(G)$, and $f(x) = x$ for all $x \neq w$.     Let $v = f(w)$.  Then $f$ is a fold if and only if $N(w) \subseteq N(v)$.  
\end{proposition}
\begin{proof}

First, suppose that $f$ is a fold, and hence  a graph  morphism.  Let $y \in N(w)$;  then $y \con w$, and so $f(y) \con f(w)$.  If $y \neq w$, then $f(y) = y$ and $f(w) = v$, so $y \con v$ and hence $y \in N(v)$.   If $y =w$ then $w \con w$ so by the looped condition for spider pair,  we assume that $f(w) \con id(w)$.  So $v \con w$ and  $w \in N(v)$.    Hence the neighbourhood condition is satisfied.  

Conversely, suppose that $f$ is a set map of vertices satisfying the neighbourhood condition.  
To show that $f$ is a morphism, we check that it preserves all connections.  If  $x,x'\in V(G)\backslash \{w\}$ and  $x\con x'$, then $f(x)=x, f(x')=x'$, and so  $f(x)\con f(x')$.  If  $y\in V(G)\backslash \{w\}$ and  $w\con y$, then $y \in N(w)  \subseteq N(v)$, so $v \con y$ and hence  $f(w)\con f(y)$.  Lastly, if  $w$ is looped, then  $w \in N(w) \subseteq N(v)$, so $v\con w$.  But then $v \in N(w) \subseteq N(v)$, and consequently $v$ must be looped as well.  Thus $f(w)\con f(w)$.

To see that $f$ is a fold, we know that if $x\in V(G)$, we have that $f(x)=x$ if and only if $x\neq w$.  So we just need to check that the extra condition on looped vertices holds.   If $w \con w$ then $w \in N(w) \subseteq N(v)$ and so   $v \con w$.  
\end{proof}

\begin{example} 
Let $X=P_2$ and let $f:G\to G$ be defined by $f(a)=a, f(b)=b, f(c)=a$.
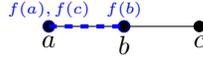
$$
\begin{tikzpicture}

\draw[fill] (3,0) circle (2pt);
\draw[blue] (3,0) circle (2.2pt);
\draw (3,0) --node[below] {$a$} (3,0);
\draw (3,0) --node[above, blue] {\tiny{$f(a), f(c)$}} (3,0);
\draw[fill] (4,0) circle (2pt);
\draw[blue] (4,0) circle (2.2pt);
\draw (4,0) --node[below] {$b$} (4,0);
\draw (4,0) --node[above, blue] {\tiny{$f(b)$}} (4,0);
\draw[fill] (5,0) circle (2pt);
\draw (5,0) --node[below] {$c$} (5,0);

\draw (3,0)--(4,0)--(5,0);
\draw[blue, ultra thick, dashed] (3,0)--(4,0);

%\draw (4, -1.1) -- node{$\Img(f)$} (4,-1.1);

\end{tikzpicture}
$$

The vertex that $f$ does not fix is $c$, and $N(c) = \{ b\} = N(a)$.  Hence the neighborhood condition of Proposition \ref{P:fold2} holds here, and this is a fold map.

\end{example}

The fact that a fold, as defined using the neighbourhood condition, gives a homotopy equivalence is proved in \cite{Kosolov3}   by looking at the polyhedral Hom complex.    Propositions \ref{P:fold} and \ref{P:fold2} offer an alternate approach which is internal to graphs.

%%%%%%%%%%%%%%%%%%%%%%%%%%%%%%%%%%%%%%%%%%%%%%%%%%%%%%%%%%%%%%%%%%%%%%%

\section{Defining a  Homotopy Category for  Finite Graphs}\label{S:HoCat}

For this section, we restrict to finite graphs and consider the full sub-category $\fGph$ of $\Gph$ consisting of graphs with a finite set of vertices.  We will be applying Proposition \ref{P:spider} to decompose homotopies as a sequence of spider moves, and so will require a finite set of vertices in our domains.  

Since our 2-cells are defined by  homotopies, known to be an equivalence relation on morphisms \cite{Docht1}, we can make the following definition.  

\begin{definition}\label{D:hocat}
We define the {\bf homotopy category} $\HoGph$ by modding  out the 2-cells in   the 2-category $\fGph$.  The objects of  $\HoGph$ are the same as the objects of $\fGph$, finite graphs, and the arrows of   $\HoGph$ are given by equivalence classes  $[f]$ of graph morphisms, where $f$ and $g$ are equivalent if they have a 2-cell between them, that is, if they are $\times$-homotopic.   This also defines a natural projection functor $\Psi:  \fGph \to \HoGph$ which takes any graph $G$ to $G$, and any morphism $f$ to its homotopy class  $[f]$.     \end{definition}  

Since all the 2-cells of $\fGph$ have become isomorphisms in $\HoGph$, the result is an ordinary 1-category.  
We will show that this is a homotopy category for $\fGph$ in the sense that it satisfies the universal property for localizing homomotopy equivalences as described in the following result.

\begin{theorem}\label{T:hocat}
Given any  functor $F:\fGph\to \mathcal{C}$ such that $F$ takes homotopy equivalences to isomorphisms,  then there is a unique functor $F':\HoGph\to\mathcal{C}$ such that $F'\Psi=F$.

$$
\begin{tikzpicture}
\node (G) at (0,2){$\fGph$};
\node (H) at (0,0){$\HoGph$};
\node (C) at (3,0){$\mathcal{C}$};

\draw[->] (G) -- node[left]{$\Psi$} (H);
\draw[->] (G) -- node[above right]{$F$} (C);
\draw[->, dashed] (H) -- node[below]{$\exists!F'$} (C);

\draw (1,2/3) --node{$=$} (1,2/3);

\end{tikzpicture}
$$

\end{theorem}
\begin{proof}
It is clear that $F':\HoGph\to \mathcal{C}$ needs to have  $F'(G)=F(G)$ for any $G\in \Obj(\HoGph)$ and $F'([f])=F(f)$ for any $[f]\in \Hom(\HoGph).$  Since $\Obj(\Gph)=\Obj(\HoGph)$, we have that $F'$ is well defined on $\Obj(\HoGph)$.  It remains to show that $F'$ is well defined on $\Hom(\HoGph)$:  that is, given $f,f'\in[f]$, we always have  $F(f)=F(f')$.  By Proposition \ref{P:spider}, it suffices to show that $F(f)=F(f')$ whenever $f,f'$ are a spider pair.

  Let $f,f':G\to H$ be a spider pair.  Then there is a vertex  $v\in V(G)$ such that $f(w)=f'(w)$ for all $w \neq v$. 
Define a new graph  $\hat{G}$ as follows:  

\begin{eqnarray*}
V(\hat{G})&=&V(G)\cup\{v^*\}.\\
E(\hat{G})&=&\begin{cases}w_1\con w_2 & \textup{ when }  w_1\con w_2\in E(G), \\ v^*\con w &\textup{ when } v \con w\in E(G)
\\ v^*\con v^* &\textup{ when } v \con v \in E(G) \end{cases}.
\end{eqnarray*}
 Thus the new vertex $v^*$ is attached to the same vertices as $v$, and is looped if and only if $v$ is looped.  
 
Let $\iota_1:G\to \hat{G}$ be the inclusion defined by $\iota_1(w)=w$ for  $w\in V(G)$.  Let $\iota_2:G\to \hat{G}$ be the inclusion defined by $\iota_2(w)=w$ for each $w\in V(G)\backslash \{v\}$ and $\iota_2(v)=v^*$. Since $N(v)=N(v^*)$ in $\hat{G}$, this is a graph morphism.

Define $\hat{f}:\hat{G}\to H$ by $$\hat{f}(w)=\begin{cases}  f(w) & \textup{ if } w\in V(G) \\ f'(v) & \textup{ if } w=v^*   \end{cases}.$$

We claim that $\hat{f}$ is a graph morphism:  suppose $w_1\con w_2\in E(\hat{G})$.  If $w_1, w_2\in V(G)$, then $\hat{f}$ agrees with $f$,  so since $f$ is a graph morphism, $\hat{f}(w_1) \con \hat{f}(w_2).  $   Now suppose that  $w_1=v^*$ and $w_2 \in V(G)$.   Then $\hat{f}(w_1) = f'(v)$ and $\hat{f}(w_2)=f(w)= f'(w)$.  Then  $f'(v) \con f'(w)$ since $f'$ is a graph morphism, so $\hat{f}(w_1) \con \hat{f}(w_2)$.    Lastly, if $w_1 = w_2 = v$ then $\hat{f}(w_i) = f'(v)$, which will be looped since $v$ was looped and $f'$ is a graph morphism.   

%We next show that $f=\hat{f}\iota_1$ and $f'=\hat{f}\iota_2$.  Notice that for any $w\in V(G)$, $$\hat{f}(\iota_1(w))=\hat{f}(w)=f(w).$$
%On the other hand for any $w\in V(G)$,

%$$\hat{f}(\iota_2(w))=\begin{cases} \hat{f}(w) & w\neq v, \\ \hat{f}(v^*) & w=v    \end{cases}=\begin{cases} f(w) & w\neq v, \\ f'(v) & w=v    \end{cases}=\begin{cases} f'(w) & w\neq v, \\ f'(v) & w=v    \end{cases}=f'(w).$$
It is clear from the definition that $f=\hat{f}\iota_1$ and $f'=\hat{f}\iota_2$.
Define $\rho:\hat{G}\to G$ by $$\rho(w)=\begin{cases} w & \textup{ if } w\in V(G) \\ v & \textup{ if } w=v^*   \end{cases}$$ This is a fold by Proposition \ref{P:fold2} since $N(v^*)=N(v)$.  Moreover, $\rho\iota_1=\rho\iota_2=\id_G$. 

So notice that $F(\id_G)=F_{\id_G}=F(\rho\iota_1)=F(\rho)F(\iota_1)$.  Similarly $F_{\id_G}=F(\rho\iota_2)=F(\rho)F(\iota_2)$.  Since $\rho$ is a homotopy equivalence, $F(\rho)$ is an isomorphism and thus $F(\iota_1)=F(\iota_2)$.

%\textcolor{blue}{Replace with ref argument} 
%By Proposition \ref{P:fold}, we know that $\rho$ and $\iota_1,\iota_2$ are homotopy inverses, so $\iota_1\rho \simeq \id_H \simeq \iota_2\rho$.    Thus $F(\iota_1\rho)=F(\iota_1)F(\rho)=\id_H=F(\iota_2)F(\rho)=F(\iota_2\rho)$.  So  $F(\rho)$ is an isomorphism and $F(\iota_1)=F(\iota_2)$.

Finally, we conclude that \begin{eqnarray*}
F(f)&=&F(\hat{f}\iota_1)\\
&=&F(\hat{f})F(\iota_1)\\
&=&F(\hat{f})F(\iota_2)\\
&=&F(\hat{f}\iota_2)\\
&=&F(f').
\end{eqnarray*}

Thus, given $f,f'\in [f]$, we have that $F'([f])=F(f)=F(f')=F'([f'])$ and $F'$ is well defined.

\end{proof}

%%%%%%%%%%%%%%%%%%%%%%%%%%%%%%%%%%%%%%%%%%%%%%%%%%%%%%%%%%%%%%%%%%%%%%%%%%
\section{A Skeleton for the Homotopy Category} \label{S:skel}

Because we have created a homotopy category without a model structure, we look for another way to describe the structure of $\HoGph$. In this section, we will show that finite stiff graphs represent all finite graphs up to homotopy.  More precisely, we  show that the finite stiff graphs form a skeleton of the homotopy category $\HoGph$ in the following sense.

\begin{definition}\label{D:skel} \cite{riehlCTIC}
A full subcategory $\sf{D}$ of a category $\sf{C}$ is a {\bf skeleton} of $\sf{C}$ provided
\begin{itemize}
    \item the inclusion $\sf{D} \into \sf{C}$ is essentially surjective, meaning that every object $C \in \sf{C}$ is isomorphic to an object $D \in \sf{D}$
    \item no two distinct objects of $\sf{D}$ are isomorphic.
\end{itemize}
\end{definition}

 In the  literature, graphs that cannot  folded are referred to as {\bf stiff} graphs \cite{GMDG} \cite{BonatoCaR}.

\begin{definition}\label{D:nerve}  We say that a graph $G$ is {\bf stiff}  if there are no two distinct vertices $v, w$ such that $N(v) \subseteq N(w)$. 

\end{definition}

\begin{example} \label{E:p1}
One large family of stiff graphs are {\bf cores} \cite{Core, HN2004}.  Since folds are graph morphisms, a core $C$ cannot admit any folds and thus must be stiff.
 Therefore complete graphs, odd cycles, and all graphs where the only endomorphisms are automorphisms are  minimal retracts.

\end{example}

\begin{example}\label{E:p2}
Another family of pleats is given by   cycles of size 6 or greater.  It is clear that $C_4$ will admit a fold, but for any greater cycle, distinct vertices can share at most 1 neighbor.  The odd cycles are covered under Example \ref{E:p1};   large even cycles are also stiff.  
\end{example}

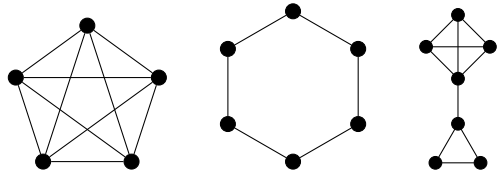
\begin{figure}[h]
$$\begin{tikzpicture}\graph[circular placement,   empty nodes, nodes={circle,draw, inner sep=2pt,minimum size=1pt, fill}] { subgraph K_n [n=5, clockwise, radius=1cm] };\end{tikzpicture}\  \  \ \ \ \ \begin{tikzpicture}\graph[circular placement,   empty nodes, nodes={circle,draw, inner sep=2pt,minimum size=1pt, fill}] { subgraph C_n [n=6, clockwise, radius=1cm] };\end{tikzpicture}\ \ \ \ \ \ 
\begin{tikzpicture}[scale=0.6]
\draw[fill] (-1/2, -0.866) circle (4pt);
\draw[fill] (1/2, -0.866) circle (4pt);
\draw[fill] (0,0) circle (4pt);
\draw[fill] (0,1) circle (4pt);
\draw[fill] (0.707,1.707) circle (4pt);
\draw[fill] (-0.707,1.707) circle (4pt);
\draw[fill] (0,1.707+0.707) circle (4pt);

\draw (0,0) -- (-1/2, -0.866) -- (1/2, -0.866) -- (0,0) -- (0,1) -- (0.707, 1.707) -- (0, 1.707+0.707) -- (-0.707, 1.707) -- (0,1);

\draw (0.707, 1.707) -- (-0.707, 1.707);

\draw (0,1) -- (0, 1.707+0.707);

\end{tikzpicture}
$$
\caption{Three examples of pleats.  On the left a core $K_5$, in the middle an even cycle $C_6$, on the right a graph that is neither a core nor even cycle.}
\end{figure}

Let $\sGph$ refer to the full subcategory of finite stiff graphs in $\HoGph$.  Thus the objects of $\sGph$ are the  finite stiff graphs, and the morphisms are homotopy classes of graph morphisms.    

\begin{theorem}\label{T:skel}  $\sGph$ is a skeleton of $\HoGph$ in the sense of Definition \ref{D:skel}.
\end{theorem}
 We will consider the two conditions of Definition \ref{D:skel} separately. 
 
\begin{proposition}\label{L:esssurj} The inclusion  $\sGph \into \HoGph$ is essentially surjective.
\end{proposition}

\begin{proof}
We proceed via induction on $n:=|V(G)|$.  Note that if $n=1$, $G$ is necessarily stiff.   Suppose $n>1$ and $G$ is not stiff, then there are distinct vertices $v, w$ such that $N(v) \subseteq N(w)$, and we can define a fold map $\rho:G \to G-\{ v\}$ which takes $v$ to $ w$ and is the identity on all other vertices.  By Propositions \ref{P:fold} and \ref{P:fold2}, this is a homotopy equivalence.  By induction, $G-\{v\}$ is homotopy equivalent to a stiff graph, and thus $G$ is as well.  

\end{proof}

To show the second condition, we note that  any sequence of folds yields  a unique graph up to isomorphism,  proved in  \cite{HN2004, GMDG, BonatoCaR} in the context of cops and robbers on graphs.  In \cite{Docht1} Proposition 6.6,  Dochterman applies this and the interpretation of homotopy in the polyhedral Hom complex to show that if $G$ and $H$ are stiff graphs then $G$ and $H$ are homotopy equivalent if and only if they are isomorphic, verifying the second condition for the skeleton.    Here we offer an alternate proof which does not use the  Hom space.

We start with the following.

\begin{lemma} \label{T:minnerve} If $G$ is stiff, then $G$ is not homotopy equivalent to any proper subgraph of itself.

\end{lemma}

\begin{proof}  

We first show that for a stiff graph $G$, the identity is not homotopic to any other endomorphism.  Suppose that $f \con id_G$.  Let  $v\in V(G)$, and  let $x \in N(v) $.  Then $f(v) \con id_G(x)$, i.e. $f(v) \con x$  so $x  \in N(f(v))$.      So $N(v) \subseteq N(f(v))$.  By the neighbourhood condition, we conclude that $v = f(v)$ and so $f = id$.  

Then, suppose that $G$ is homotopy equivalent to a subgraph of itself  $H$.  So we have $f:  G \to H$ and $g:  H \to G$ such that $gf $ is homotopic to  $id_G$.   Then $gf$ must actually be the identity on $G$.  Hence $G$ is isomorphic to $H$.  
\end{proof}

\begin{lemma}\label{L:image}
If $f: G\to G$ such that $f \con id_G$ then $G$ is homotopy equivalent to $Im(f)$.  
\end{lemma}
\begin{proof}
Let $\iota$ denote the inclusion map $Im(f) \to G$.  Then  $\iota f=f$ which is homotopic to $id_G$.  We need to show that $f \iota$ is homotopic to $\id_H$ where $H = Im(f)$.    Suppose that $v\con w \in E(H)$.   Then by our definition of $Im(f)$ in Definition \ref{D:im}, this edge is the image of an edge $v'\con w' \in E(G)$, where  $f(v')=v, f(w')=w$.    Since $f \con id_G$, we know that  $v'\con f(w') \in E(G)$ and  therefore $f(v')\con  f(f(w')) \in E(H)$.  So $v \con f \iota(w)$ whenever $v \con w$, and so $id_H \con f\iota$.     \end{proof}

\begin{theorem}[\cite{Docht1}, Proposition 6.6]\label{T:minrep}
If $G, H$ are finite stiff graphs which are homotopy equivalent, then $G$ and $H$ are isomorphic.   \end{theorem}
\begin{proof}

Suppose we have graph morphisms $f:G\to H, g:H\to G$ such that $g f\simeq \id_G $ and $ f g \simeq \id_{H}$.  Thus by Proposition \ref{P:spider} we have a sequence of maps $id_{G}, k_1, k_2, \dots, k_n = gf$ such that each successive pair is a spider pair.   So by  Lemma \ref{L:image} $Im(g f)$ is homotopy equivalent to $G$.  Since $G$ is stiff, it follows that $Im(g f)=G$.  Similarly $Im(f g)=H$, and $f,g$ are isomorphisms.

\end{proof}

This completes the proof of Theorem \ref{T:skel}.

\begin{obs}
Any graph which is not  stiff may be folded.  Thus, we obtain a homotopy equivalent graph by continuous applying folds as in Lemma \ref{L:esssurj}. A  consequence of Theorem \ref{T:minrep} is that one may apply these folds in any arbitrary fashion, and the resulting stiff graphs will be isomorphic.  See Figure 6.2 below.

\begin{figure}[h]
$$\begin{tikzpicture}
\draw[fill] (0,0) circle (2pt);
\draw[fill] (.5,.5) circle (2pt);
\draw[fill] (-.5,.5) circle (2pt);
\draw[fill] (0,1) circle (2pt);
\draw[fill] (0,1.707) circle (2pt);

\draw (0,1.707) -- (0,1);
\draw (-.5,.5) -- (0,1);
\draw (.5,.5) -- (0,1);
\draw (-.5,.5) -- (0,0);
\draw (.5,.5) -- (0,0);

\draw[ultra thick, red] (.5,.5) -- (0,0);

\end{tikzpicture}
\ \ 
\begin{tikzpicture}
\draw[fill] (0,0) circle (2pt);
\draw[fill] (.5,.5) circle (2pt);
\draw[fill] (-.5,.5) circle (2pt);
\draw[fill] (0,1) circle (2pt);
%\draw[fill] (0,1.707) circle (2pt);

%\draw (0,1.707) -- (0,1);
\draw (-.5,.5) -- (0,1);
\draw (.5,.5) -- (0,1);
\draw (-.5,.5) -- (0,0);
\draw (.5,.5) -- (0,0);

\draw[ultra thick, red] (.5,.5) -- (0,0);
\draw (0,0) -- (0,0);
\draw (.5,.5) -- (.5,.5);
\draw (-.5,.5) -- (-.5,.5);

\end{tikzpicture}
\ \ 
\begin{tikzpicture}
\draw[fill] (0,0) circle (2pt);
\draw[fill] (.5,.5) circle (2pt);
%\draw[fill] (-.5,.5) circle (2pt);
\draw[fill] (0,1) circle (2pt);
%\draw[fill] (0,1.707) circle (2pt);

%\draw (0,1.707) -- (0,1);
%\draw (-.5,.5) -- (0,1);
\draw (.5,.5) -- (0,1);
%\draw (-.5,.5) -- (0,0);
\draw (.5,.5) -- (0,0);

\draw[ultra thick, red] (.5,.5) -- (0,0);
\draw (0,0) -- (0,0);
\draw (.5,.5) -- (.5,.5);
\draw (-.5,.5) -- (-.5,.5);

\end{tikzpicture}
\ \ 
\begin{tikzpicture}
\draw[fill] (0,0) circle (2pt);
\draw[fill] (.5,.5) circle (2pt);
%\draw[fill] (-.5,.5) circle (2pt);
%\draw[fill] (0,1) circle (2pt);
%\draw[fill] (0,1.707) circle (2pt);

%\draw (0,1.707) -- (0,1);
%\draw (-.5,.5) -- (0,1);
%\draw (.5,.5) -- (0,1);
%\draw (-.5,.5) -- (0,0);
\draw (.5,.5) -- (0,0);

\draw[ultra thick, red] (.5,.5) -- (0,0);
\draw (0,0) -- (0,0);
\draw (.5,.5) -- (.5,.5);
\draw (-.5,.5) -- (-.5,.5);

\end{tikzpicture}$$

$$\begin{tikzpicture}
\draw[fill] (0,0) circle (2pt);
\draw[fill] (.5,.5) circle (2pt);
\draw[fill] (-.5,.5) circle (2pt);
\draw[fill] (0,1) circle (2pt);
\draw[fill] (0,1.707) circle (2pt);

\draw (0,1.707) -- (0,1);
\draw (-.5,.5) -- (0,1);
\draw (.5,.5) -- (0,1);
\draw (-.5,.5) -- (0,0);
\draw (.5,.5) -- (0,0);

\draw[ultra thick, blue] (0,1.707) -- (0,1);

\end{tikzpicture}
\ \ 
\begin{tikzpicture}
%\draw[fill] (0,0) circle (2pt);
\draw[fill] (.5,.5) circle (2pt);
\draw[fill] (-.5,.5) circle (2pt);
\draw[fill] (0,1) circle (2pt);
\draw[fill] (0,1.707) circle (2pt);

\draw (0,1.707) -- (0,1);
\draw (-.5,.5) -- (0,1);
\draw (.5,.5) -- (0,1);
%\draw (-.5,.5) -- (0,0);
%\draw (.5,.5) -- (0,0);

\draw[ultra thick, blue] (0,1.707) -- (0,1);
\draw (0,0) -- (0,0);
\draw (.5,.5) -- (.5,.5);
\draw (-.5,.5) -- (-.5,.5);

\end{tikzpicture}
\ \ 
\begin{tikzpicture}
%\draw[fill] (0,0) circle (2pt);
\draw[fill] (.5,.5) circle (2pt);
%\draw[fill] (-.5,.5) circle (2pt);
\draw[fill] (0,1) circle (2pt);
\draw[fill] (0,1.707) circle (2pt);

\draw (0,1.707) -- (0,1);
%\draw (-.5,.5) -- (0,1);
\draw (.5,.5) -- (0,1);
%\draw (-.5,.5) -- (0,0);
%\draw (.5,.5) -- (0,0);

\draw[ultra thick, blue] (0,1.707) -- (0,1);
\draw (0,0) -- (0,0);
\draw (.5,.5) -- (.5,.5);
\draw (-.5,.5) -- (-.5,.5);
\end{tikzpicture}
\ \ 
\begin{tikzpicture}
%\draw[fill] (0,0) circle (2pt);
%\draw[fill] (.5,.5) circle (2pt);
%\draw[fill] (-.5,.5) circle (2pt);
\draw[fill] (0,1) circle (2pt);
\draw[fill] (0,1.707) circle (2pt);

\draw (0,1.707) -- (0,1);
%\draw (-.5,.5) -- (0,1);
%\draw (.5,.5) -- (0,1);
%\draw (-.5,.5) -- (0,0);
%\draw (.5,.5) -- (0,0);

\draw[ultra thick, blue] (0,1.707) -- (0,1);
\draw (0,0) -- (0,0);
\draw (.5,.5) -- (.5,.5);
\draw (-.5,.5) -- (-.5,.5);
\end{tikzpicture}$$
\caption{Any series of folds of will eventually terminate with a subgraph isomorphic to $K_2$, although not necessarily the same subgraph.}
\end{figure}
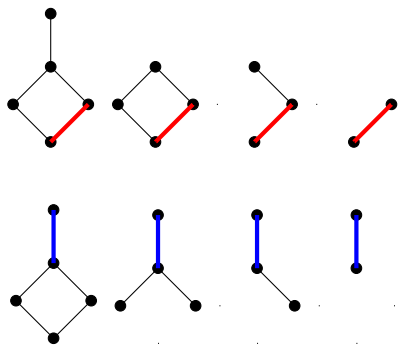

\end{obs}

\section*{Acknowledgements}
The authors are grateful to Dr.\ Demitri Plessas for his previous forays into categorical graph theory, and for his feedback.  We also want to thank Dr.\ Jeffery Johnson for helping us with some of the terminology in this paper.  Lastly, we want to thank the referees who pointed us in the direction of the relevant literature.  

\bibliographystyle{spmpsci}      % mathematics and physical sciences
\bibliography{HCG}   % name your BibTeX data base

\begin{bibdiv}
\begin{biblist}

\bib{HomTG}{article}{
      author={Babson, Eric},
      author={Barcelo, H'{e}l\`{e}ne},
      author={de~Loungeville, Marke},
      author={Laubenbacher, Reinhard},
       title={Homotopy theory of graphs},
        date={2006},
     journal={Journal of Alegbraic Combinatorics},
      volume={24},
      number={1},
       pages={31\ndash 44},
}

\bib{ProofLovasz}{article}{
      author={Babson, Eric},
      author={Kozlov, Dmitry~N.},
       title={Proof of the {L}ov\`{a}sz conjecture},
        date={2007},
     journal={Annals of Mathematics},
      volume={165},
      number={3},
       pages={965\ndash 1007},
}

\bib{BonatoCaR}{book}{
      author={Bonato, A.},
      author={Nowakowski, R.J.},
       title={The game of cops and robbers on graphs},
      series={Student mathematical library},
   publisher={American Mathematical Society},
        date={2010},
}

\bib{Bondy}{article}{
      author={Bondy, JA},
      author={Murty, USR},
       title={Graph theory. 2008},
        date={2008},
     journal={Grad. Texts Math},
}

\bib{GMDG}{article}{
      author={Brightwell, Graham~R.},
      author={Winkler, Peter},
       title={Gibbs measures and dismantlable graphs},
        date={2000},
     journal={Journal of Combinatorial Theory, Series B},
      volume={78},
      number={1},
       pages={141\ndash 66},
}

\bib{Docht1}{article}{
      author={Dochtermann, Anton},
       title={Hom complexes and homotopy theory in the category of graphs},
        date={2009Febuary},
     journal={European Journal of Combinatorics},
      volume={30},
      number={2},
       pages={490\ndash 509},
}

\bib{Docht2}{article}{
      author={Dochtermann, Anton},
       title={Homotopy groups of hom complexes of graphs},
        date={2009},
     journal={Journal of Combinatorial Theory, Series A},
      volume={116},
      number={1},
       pages={180\ndash 194},
}

\bib{Droz}{article}{
      author={Droz, Jean-Marie},
       title={Quillen model structions on the category of graphs},
        date={2012},
     journal={Homology, Homotopy and Applications},
      volume={14},
      number={2},
       pages={265\ndash 284},
}

\bib{MoreFolding}{article}{
      author={Fieux, E.},
      author={Lacaze, J.},
       title={Foldings in graphs and relations with simplicial complexes and
  posets},
        date={2012September},
     journal={Discrete Mathematics},
      volume={312},
      number={17},
       pages={2639\ndash 2651},
}

\bib{AGT}{book}{
      author={Godsil, Chris},
      author={Royle, Gordon},
       title={Algebraic {G}raph {T}heory},
   publisher={Springer},
        date={2001},
}

\bib{NoModel}{misc}{
      author={Goyal, Shuchita},
      author={Santhanam, Rekha},
       title={({L}ack of) {M}odel structures on the category of graphs},
        date={2005},
         url={https://arxiv.org/abs/1902.09182},
}

\bib{Hatcher}{book}{
      author={Hatcher, Allen},
       title={Algebraic topology},
   publisher={Cambridge University Press},
        date={2001},
}

\bib{HN2004}{book}{
      author={Hell, Pavol},
      author={Ne{\v{s}}et{\v{r}}il, Jaroslav},
       title={Graphs and homomorphisms},
      series={Oxford Lecture Series in Mathematics and its Applications},
   publisher={Oxford University Press},
     address={Oxford},
        date={2004},
      volume={28},
}

\bib{Core}{article}{
      author={Hell, Pavol},
      author={Nesetril, Jarslave},
       title={The core of a graph},
        date={1992November},
     journal={Discrete Mathematics},
      volume={109},
       pages={117\ndash 126},
}

\bib{Kosolov1}{misc}{
      author={Kozlov, Dmitry~N.},
       title={Chromatic numbers, morphism complexes, and {S}tiefel-{W}hitney
  characteristic classes},
        date={2005},
         url={https://arxiv.org/abs/math/0505563},
}

\bib{Kosolov2}{article}{
      author={Kozlov, Dmitry~N.},
       title={Collapsing along monotone poset maps},
        date={2006},
     journal={International Journal of Mathematics},
      volume={8},
}

\bib{Kosolov3}{article}{
      author={Kozlov, Dmitry~N.},
       title={Simple homotopy types of {H}om-complexes, neighborhood complexes,
  {L}ov\`{a}sz complexes, and atom crosscut complexes},
        date={2006},
     journal={Topology Applications},
      volume={14},
       pages={2445\ndash 2454},
}

\bib{KosolovShort}{article}{
      author={Kozlov, Dmitry~N.},
       title={A simple proof for folds on both sides in complexes of graph
  homomorphisms},
        date={2006},
     journal={Proceedings of the American Mathematical Society},
      volume={134},
      number={5},
       pages={1265\ndash 1270},
}

\bib{Mac}{book}{
      author={Mac~Lane, Saunders},
       title={Categories for the working mathematician},
     edition={Second},
      series={Graduate Texts in Mathematics},
   publisher={Springer-Verlag},
     address={New York},
        date={1998},
      volume={5},
}

\bib{BoxHomotopy}{article}{
      author={Matsushita, Takahiro},
       title={Box complexes and homotopy theory of graphs},
        date={2017},
     journal={Homology, Homotopy and Applications},
      volume={19},
      number={2},
       pages={175\ndash 197},
}

\bib{Demitri}{thesis}{
      author={Plessas, Demitri},
       title={{T}he {C}ategories of {G}raphs},
        type={Ph.D. Thesis},
        date={2012},
}

\bib{riehlCTIC}{book}{
      author={Riehl, E.},
       title={Category theory in context},
      series={Aurora: Dover Modern Math Originals},
   publisher={Dover Publications},
        date={2017},
        ISBN={9780486820804},
}

\end{biblist}
\end{bibdiv}

\end{document}